\documentclass[12pt]{amsart}

\usepackage[margin=1in]{geometry}
\usepackage{xcolor}
\usepackage{amsmath,amsthm,amssymb,mathtools,mathrsfs}
\usepackage{graphicx}
\usepackage[shortlabels]{enumitem}
\usepackage{shuffle}
\usepackage{multicol} 
\usepackage{commath}
\usepackage{ytableau}
\usepackage{caption,subcaption}
\usepackage{hyperref,cleveref}
\usepackage{comment,soul}
\usepackage{bm,standalone}
 \usepackage{tikz}
\usetikzlibrary{calc}
\usepackage{tikz-cd}

\newcommand{\Z}{\mathbb{Z}}

\newcommand{\lp}{\left(}
\newcommand{\rp}{\right)}

\renewcommand{\phi}{\varphi}
\newcommand{\eps}{\varepsilon}

\DeclareMathOperator{\inc}{Inc}

\newtheorem{theorem}{Theorem}[section]
\newtheorem{cor}[theorem]{Corollary}
\newtheorem{conjecture}[theorem]{Conjecture}
\newtheorem{prop}[theorem]{Proposition}
\newtheorem{lemma}[theorem]{Lemma}
\newtheorem{question}[theorem]{Question}

\theoremstyle{definition}
\newtheorem{definition}[theorem]{Definition}
\newtheorem{remark}[theorem]{Remark}

\usepackage{hyperref}

\renewcommand{\a}{\ora{a}}

 \usepackage{tikz}
\usetikzlibrary{calc}

\definecolor{lavender}{RGB}{176,102,212} 
\definecolor{john}{RGB}{255,137,54}

\newcommand{\cX}{\mathcal{X}}
\newcommand{\Tad}{\mathrm{Tad}} 
\newcommand{\supp}{\mathrm{supp}}

\allowdisplaybreaks

\begin{document}

\title[The $e$-positivity of twinned paths and cycles]{The $e$-positivity of the chromatic symmetric function for twinned paths and cycles}
\author[Banaian et al.]{Esther Banaian}\address{Department of Mathematics, Aarhus University, Aarhus DK}
\email{\url{banaian@math.au.dk}}
\author[]{Kyle Celano}\address{Department of Mathematics, Wake Forest University, NC}
\email{\url{celanok@wfu.edu}}
\author[]{Megan Chang-Lee}\address{Department of Mathematics, Brown University, Providence, RI}
\email{\url{megan_chang-lee@brown.edu}}

\author[]{Laura Colmenarejo}\address{Department of Mathematics, North Carolina State University, Raleigh, NC}\email{\url{lcolmen@ncsu.edu}}
\author[]{Owen Goff}\address{Department of Mathematics, University of Wisconsin, Madison, WI}\email{\url{ogoff@wisc.edu}}
\author[]{Jamie Kimble}\address{Department of Mathematics, Michigan State University, East Lansing, MI}\email{kimblej2@msu.edu}
\author[]{Lauren Kimpel}\address{Department of Mathematics, Johns Hopkins University, Baltimore, MD}\email{lkimpel1@jhu.edu}
\author[]{John Lentfer}\address{Department of Mathematics, University of California, Berkeley, CA}\email{jlentfer@berkeley.edu}
\author[]{Jinting Liang}\address{Department of Mathematics, Michigan State University, East Lansing, MI}\email{liangj26@msu.edu}
\author[]{Sheila Sundaram}
\address{School of Mathematics and Statistics, University of Minnesota, Minneapolis, MN}
\email{shsund@umn.edu}



\begin{abstract} The operation of twinning a graph at a vertex was introduced by Foley, Ho\`ang, and Merkel (2019), who conjectured that twinning preserves $e$-positivity of the chromatic symmetric function. A counterexample to this conjecture was given by Li, Li, Wang, and Yang (2021).  In this paper, we prove that $e$-positivity is preserved by the twinning operation on cycles, by giving an $e$-positive generating function for the chromatic symmetric function, as well as an $e$-positive recurrence.  We derive similar $e$-positive generating functions and recurrences for twins of paths.  Our methods make use of the important triple deletion formulas of Orellana and Scott (2014), as well as new symmetric function identities. 
\end{abstract}

\maketitle

\tableofcontents

\section{Introduction}
The \textit{chromatic symmetric function} of a graph $G=(V,E)$ is defined by Stanley~\cite{stanley_1995} to be
\[X_{G}  (\mathbf{x})=\sum_{\kappa}\prod_{v\in V}x_{\kappa(v)},\]
where the sum is over all proper colorings $\kappa:V\to \Z_{>0}$ of $G$ by positive integers. The goal of this paper is to study the effect that a small change to the graph $G$ has on  $X_G(\mathbf{x})$. Specifically, we look at the change in $X_G(\mathbf{x})$ when one \textit{twins} (or \textit{clones}) a vertex $v$ of a graph $G$, that is, when one adds a vertex $v'$ incident to $v$ and all its neighbors, to produce a new graph $G_v$. Precise definitions of this operation and related terms appear in~\Cref{sec:prelim}. 

\begin{question}\label{question: twinning}
    For a given vertex $v$ of a graph $G$, how are the polynomials $X_{G_v}(\mathbf{x})$ and $X_G(\mathbf{x})$ related? 
\end{question}
In the seminal paper~\cite{stanley_1995}, 
Stanley proved that the chromatic symmetric functions of paths and cycles are $e$-positive, that is, their expansion in the basis of elementary symmetric functions has nonnegative coefficients.  As observed in~\cite{stanley_1995}, the result for paths is originally due to Carlitz, Scoville, and Vaughan in a different context~\cite[p.242]{CarScoVau1976}.
More generally, spurred by the following conjecture of Stanley and Stembridge,  much of the research on the chromatic symmetric function has centered around the incomparability graph $\inc(P)$ of a $(3+1)$-free poset $P$, defined as a poset containing no induced subposet isomorphic to the disjoint union of a $3$-chain and a $1$-chain.

\begin{conjecture}[\cite{stanley_1995,stanstem}]
    If $P$ is a $(3+1)$-free poset, then $X_{\inc(P)}(\mathbf{x})$ is $e$-positive.
\end{conjecture}

To \textit{twin} a poset $P$ at a vertex $v$, producing $P_v$, is to add an element $v'$ such that $v'$ is incomparable to $w$ if and only if either $w=v$ or $w$ is incomparable to $v$.
Note that if $G=\inc(P)$, then $\inc(P_v)=G_v$. This next simple lemma is the main motivation for considering the twinning operation. Its proof is immediate from the fact that if $u<v$, then  $u<v'$, and if $v<w$ then $v'<w$.
\begin{lemma}
    The twin of a $(3+1)$-free poset is $(3+1)$-free.
\end{lemma}

One can therefore make a weakened version of the Stanley--Stembridge conjecture, first appearing in the work of Foley, Ho{\`a}ng, and Merkel~\cite{foley2018classes}.

\begin{conjecture}[\cite{foley2018classes}]\label{conjecture: 3+1 twinning}
    If $P$ is $(3+1)$-free and $X_{\inc(P)}(\mathbf{x})$ is $e$-positive, then $X_{\inc(P_v)}(\mathbf{x})$ is $e$-positive for any $v\in P$.
\end{conjecture}
\begin{remark}
    Li, Li, Wang, and Yang~\cite[Theorem 3.6]{Li-Li-Wang-Yang} prove that the twinning operation on graphs does not always preserve $e$-positivity.  They give an example of a graph $G$~\cite[Theorem 4.1]{Li-Li-Wang-Yang} that is not an incomparability graph of a $(3+1)$-free poset, but whose chromatic symmetric function $X_G(\mathbf{x})$ is $e$-positive, and show that for a certain 
    vertex $v$ of $G$, the chromatic symmetric function for the twinned graph $G_v$ 
     does not expand positively even in the Schur basis, and so it cannot be $e$-positive. 
     This suggests that there is something special about the twinning operation on $(3+1)$-free posets.
\end{remark}
If~\Cref{conjecture: 3+1 twinning} is solved, then the Stanley--Stembridge conjecture can be reduced to the following conjecture. We say that a graph $H$ is \textit{twin-free} if there is no graph $G$ with vertex $v$  such that $H=G_v$.
\begin{conjecture}
    If $P$ is $(3+1)$-free and twin-free, then $X_{\inc(P)}(\mathbf{x})$ is $e$-positive.
\end{conjecture}

Other than~\cite{foley2018classes,Li-Li-Wang-Yang}, twinning does not seem to have been studied as an operation in its own right.
Rather, there has been incidental progress over the years. In 2001,  Gebhard and Sagan~\cite{gebhard_sagan_2001} introduced the stronger notion of 
$(e)$-positivity of chromatic symmetric functions in noncommuting variables. This implies  $e$-positivity for chains of complete graphs~\cite[Corollary 7.7]{gebhard_sagan_2001}, and includes twins of paths as a special case.  Later, Dahlberg and van Willigenburg~\cite{DahlbergVanWilligenburg2018} gave a direct proof of $e$-positivity for \textit{lollipop} graphs, which are a special case of~\cite[Corollary 7.7] {gebhard_sagan_2001}, and which again include paths twinned at a leaf.

Throughout the paper, we refer 
   to a property of the \emph{chromatic symmetric function of the graph $G$} as being a property of  \emph{the graph $G$}. For instance, we use interchangeably the phrases ``the generating function of the chromatic symmetric function of a graph'' and ``the generating function of a graph''. Similarly, we use ``the chromatic symmetric function of the graph $G$ is $e$-positive'' and ``the graph $G$ is $e$-positive'' interchangeably. 

\subsection{Main results}
This paper studies the effect of twinning on the $e$-expansions of the chromatic symmetric function of certain graphs. We specifically look at twins of path and cycle graphs, a few of which are shown in~\Cref{fig: twin examples}. A summary of our progress on~\Cref{question: twinning} follows.

Our first main contribution is a series of explicit $e$-positive formulas for the generating function 
of  the following families  
of twinned graphs:
\begin{enumerate}
\item The path twinned at one leaf (\Cref{prop:epos-via-gf-twin-pathS-leaf})
    \item The path twinned at both leaves (\Cref{thm:gf-double-twinS-deg-ge7-len3})
    \item The path twinned at an interior vertex (\Cref{thm:cX_P-f_ell-formula})
    \item The cycle twinned at a vertex (\Cref{thm:twin-cycle-genfunction-rational-expression})
\end{enumerate}
The fourth family, examined in detail in~\Cref{sec:gf-twin-cycle}, and culminating in~\Cref{thm:twin-cycle-genfunction-rational-expression}, was not known to be $e$-positive until now.
The first three families appear in~\cite{gebhard_sagan_2001} 
and were shown to have the stronger property of  $(e)$-positivity of their chromatic symmetric functions with noncommuting variables~\cite[Theorems 7.6 and 7.8]{gebhard_sagan_2001}. 
The  $e$-positivity of the first graph was later re-established directly in~\cite{DahlbergVanWilligenburg2018}.

However, the explicitly $e$-positive  expressions for the generating functions that we give in~\Cref{prop:epos-via-gf-twin-pathS-leaf},~\Cref{thm:gf-double-twinS-deg-ge7-len3} and~\Cref{thm:cX_P-f_ell-formula}  are new. Also, in~\Cref{cor:new-gf-PATHSS} we provide a new $e$-positive expression for the generating function of the path that isolates the terms containing $e_1$. 
Our derivations make crucial use of the triple deletion formula of Orellana and Scott~\cite{ORELLANA20141}.  

For all but the third family, the expression we obtain for the generating function has the form
\[\sum_{n\geq 0}X_{G_n}z^n=\frac{f_G}{1-\sum_{i\geq 2}(i-1)e_iz^i}+h_G\]
where $f_G$ and $h_G$ are some $e$-positive functions depending on the family and $h_G$ has finite degree. The third family has $h_G$ with infinite degree. Note that the denominator in the rational expression above coincides with the denominator in the generating function for both the path and the cycle (see~\Cref{thm:RPS-gf-paths-cycles}). This allows us to easily obtain explicit formulas for the coefficients of the elementary symmetric functions (see~\Cref{cor:e-coefficents-of-twinned-path},~\Cref{cor:double-twin-path-leaveS-coeffS},~\Cref{cor:e-coefficients-twinned-cycle}). We state our formulas for the coefficients using a new statistic $\eps(\lambda)$ associated with a partition $\lambda$  (see~\Cref{subsec: epsilon-lambda}). This statistic appears naturally when computing the $e$-coefficients of the path and cycle, and appears to be of independent interest. 

The identities presented in~\Cref{sec:Sym-fn-identitieS}, particularly in~\Cref{lem:Ez-identitieS} and its proof,  are the starting point for our $e$-positivity results. They also seem interesting in their own right.

Our second main contribution is an $e$-positive recurrence relation for each of the families listed above, as well as a graph appearing in the computation for the twinned cycle that we call the \textit{moose graph}, which has been shown to be $e$-positive as a special case of hat graphs~\cite[Theorem 3.9]{WangZhou2024}.
\begin{enumerate}
    \item The path twinned at one leaf (\Cref{SHORTprf-epoS-twin-path-leaf})
    \item The path twinned at both leaves (\Cref{SHORTprf-epoS-twin-path-both-leaves})
    \item The path twinned at an interior vertex (\Cref{thm:epoS-twin-pathS-recurrence})
    \item The cycle twinned at a vertex (\Cref{thm:SHORTprf-epoS-twin-cycleS-recurrence})
    \item The moose graph (\Cref{thm:MooseRecurrence})
\end{enumerate}

\section{Preliminaries}\label{sec:prelim}

In this section, we define the basic notions used throughout the paper, as well as discuss previous results. We assume a familiarity with symmetric functions as in~\cite[Chapter 7]{stanley_1999} or~\cite{macdonald1998symmetric}.

A \textit{graph} $G$ is a pair of sets $(V,E)$ where $V$ is the set of \textit{vertices} and $E$ is a set of $2$-element subsets of vertices, called \textit{edges}. We denote edges by $\{u,v\}$ or simply by $uv$. We assume that $V$ and $E$ are both finite, 
and that the graph is simple (i.e., there are no loops and no multiple edges). A \textit{leaf} of a graph is a vertex contained within exactly one edge. An \textit{internal vertex} is a vertex contained within at least two edges. Two  graphs that are important for this paper are the \textit{path} $P_n$, which has vertex set $V=[n]=\{1,\dots,n\}$ and edge set $E=\{\{i,i+1\}\mid i\in [n-1]\}$, and the cycle $C_n$, which also has vertex set $V=[n]$ and edge set $E=\{\{i,i+1\}\mid i\in [n-1]\}\cup \{1,n\}$. We illustrate them in~\Cref{fig:path-and-cycle-intro}. 

\begin{figure}[ht]
    \centering
    \begin{subfigure}[b]{.4\textwidth}
        \centering
\begin{tikzpicture}[scale = 0.7]
\node(1) at (0,0){$\bullet$};
\node(l-1) at (1.5,0){$\bullet$};
\node(l) at (3,0){$\bullet$};
\node(l+1) at (4.5,0){$\bullet$};
\node(dots2) at (5.75,0){$\cdots$};
\node(n) at (7,0){$\bullet$};
\draw(0,0) -- (1.5,0);
\draw(1.5,0) -- (3,0);
\draw(3,0) -- (4.5,0);
\draw(4.5,0) -- (dots2);
\draw(dots2) -- (7,0);
\end{tikzpicture}
        \caption{Path $P_n$}
    \end{subfigure}
    \begin{subfigure}[b]{.4\textwidth}
    \centering
  \begin{tikzpicture}[scale = 0.7]
\draw(0,2) -- (1.4,1.4);
\draw(1.4,1.4) -- (2,0);
\draw(2,0) -- (1.4,-1.4);
\draw(1.4,-1.4) -- (0.5,-2);
\draw(-0.6,-2) --(-1.4,-1.4);
\draw(-1.4,-1.4) -- (-2,0);
\draw(-2,0) -- (-1.4,1.4);
\draw(-1.4,1.4) -- (0,2);
\node(1) at (0,2){$\bullet$};
\node(2) at (1.4,1.4){$\bullet$};
\node(3) at (2,0){$\bullet$};
\node(4) at (1.4,-1.4){$\bullet$};
\node(5) at (0,-2.1){$\cdots$};
\node(8) at (-1.4,1.4){$\bullet$};
\node(7) at (-2,0){$\bullet$};
\node(6) at (-1.4,-1.4){$\bullet$};
\end{tikzpicture}  
\caption{Cycle $C_n$}
\end{subfigure}
    \caption{The path and cycle graphs.}
    \label{fig:path-and-cycle-intro}
\end{figure}
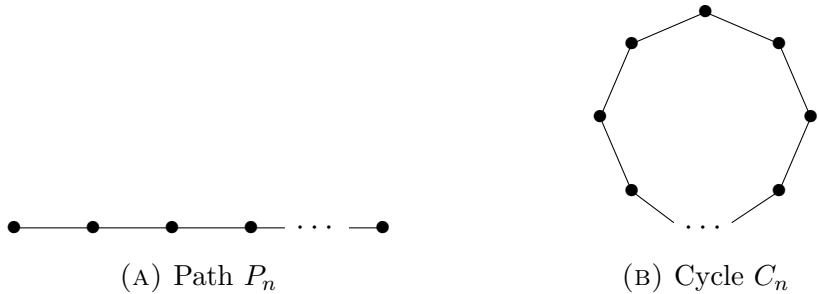

A \textit{proper coloring} of a graph $G=(V,E)$ is a function $\kappa:V\to \Z_{>0}$ such that if $uv\in E$, then $\kappa(u)\neq \kappa(v)$. Let $\mathbf{x}=(x_1,x_2,\dots)$ be an infinite set of commuting variables. The \textit{chromatic symmetric function} of a graph $G=(V,E) $ is defined to be
\[X_G:=X_G(\mathbf{x})=\sum_{\kappa} \prod_{v\in V} x_{\kappa(v)},\]
where the sum is over all proper colorings $\kappa$ of $G$.
This symmetric function was first introduced by Stanley in~\cite{stanley_1995} and has been studied by numerous authors since then. One central goal has been to characterize  graphs $G$ for which $X_G$ is $e$-positive, i.e., $X_G$ has nonnegative coefficients in the elementary symmetric function basis. We give an overview of the previous $e$-positivity results in~\Cref{subsection: e-positivity background}.

\subsection{Graph operations}
We begin by defining the two operations on graphs that appear in 
this paper.

Given an edge $\epsilon$ of a graph $G$, the \textit{deletion of $\epsilon$ in $G$} is the graph, denoted by $G-\epsilon$, obtained by removing the edge $\epsilon$ from $G$. The following formulas of Orellana and Scott are used extensively in our arguments and we refer to them as the \textit{triple deletion arguments}. 

\begin{prop}[Triple Deletion Formula {\cite[Theorem 3.1]{ORELLANA20141}}]\label{prop:Triple-deletion}
    Let $G$ be any graph. Suppose edges $\epsilon_1,\epsilon_2,\epsilon_3$ form a triangle in $G$. Then,
    \[X_G =X_{G-\epsilon_1} +X_{G-\epsilon_2} -X_{G-\{\epsilon_1,\epsilon_2\}} .\]
\end{prop}

Notice that~\Cref{prop:Triple-deletion} requires the graph to contain a triangle. However, one can use this formula to derive other relations for graphs that do not necessarily contain a triangle. An example of such a relation is the following. 
\begin{cor}[{\cite[Corollary 3.2]{ORELLANA20141}}]
\label{cor:Triple-deletion-almost-triangles}
Let $\epsilon_1=v v_1\in E$, $\epsilon_2=v v_2\in E$ and suppose $\epsilon_3=v_1v_2\not\in E$. Then 
\[X_{G} =X_{(G-\eps_1)\cup \epsilon_3} +X_{G-\epsilon_2} -X_{(G-\{\epsilon_1,\epsilon_2\})\cup \epsilon_3} .\]
\end{cor}

We now introduce the main operation studied in this paper.
\begin{definition}
Given a graph $G$ and a vertex $v$, the \textit{twin of $G$ at $v$} is the graph, denoted by $G_v$, obtained by adding a new vertex $v'$ and connecting $v'$ to $v$ and to all of its neighbors. We refer to this operation as the \textit{twinning} of a graph and to the resulting graph $G_v$ as the \textit{twinned graph}.  By extension, $G_{v,w}$ denotes the graph $G$ twinned at the vertices $v$ and $w$ in succession.
\end{definition} 

A simple example is the complete graph $G=K_n$ on $n$ vertices. For any vertex $v$, $(K_n)_v$ is the complete graph $K_{n+1}$.
We illustrate in~\Cref{fig: twin examples} the twinned path at a leaf and at an interior vertex, and the twinned cycle.

\begin{figure}[ht]
    \centering
        \begin{subfigure}[b]{.4\textwidth}
        \centering
\begin{tikzpicture}[scale = 0.7]
\node(dots1) at (1.25,0){$\cdots$};
\node(1) at (0,0){$\bullet$};
\node(v) at (3.5,0){$\bullet$};
\node(u) at (2.5,0){$\bullet$};
\node(v') at (3.5,1){$\bullet$};
\node(w) at (4.5,0){$\bullet$};
\node[below] at (v){$v$};
\node[below] at (u){$u$};
\node[below] at (w){$w$};
\node[above] at (v'){$v'$};
\node(dots2) at (5.75,0){$\cdots$};
\node(n) at (7,0){$\bullet$};
\draw (3.5,0) -- (3.5,1);
\draw(0,0) -- (dots1);
\draw (dots1) -- (dots2);
\draw(dots2) -- (7,0);
\draw (2.5,0) -- (3.5,1) -- (4.5,0);
\node(dots1) at (3.25,4){$\cdots$};
\node(1) at (2,4){$\bullet$};
\node(v) at (0,3){$\bullet$};
\node(v') at (0,5){$\bullet$};
\node(w) at (1,4){$\bullet$};
\node(middle) at (4.5,4){$\bullet$};
\node[below] at (v){$v$};
\node[below] at (w){$w$};
\node[above] at (v'){$v'$};
\node(n) at (5.75,4){$\bullet$};
\draw(1,4) -- (dots1);
\draw(dots1) -- (5.75,4);
\draw (0,3) -- (1,4) -- (0,5) -- (0,3);
\end{tikzpicture}
        \caption{Twin paths $P_{n,v}$}
    \end{subfigure}
        \begin{subfigure}[b]{0.4\textwidth}
    \centering
\begin{tikzpicture}[scale = 0.7]
\draw(0,2.7) -- (0,1.7);
\draw(0,2.7) -- (1.4,1.4);
\draw(0,1.7) -- (1.4,1.4);
\draw(1.4,1.4) -- (2,0);
\draw(2,0) -- (1.4,-1.4);
\draw(1.4,-1.4) -- (0.5,-2);
\draw(-0.6,-2) --(-1.4,-1.4);
\draw(-1.4,-1.4) -- (-2,0);
\draw(-2,0) -- (-1.4,1.4);
\draw(-1.4,1.4) -- (0,2.7);
\draw(-1.4,1.4) -- (0,1.7);
\node(1) at (0,2.7){$\bullet$};
\node(1twin) at (0,1.7){$\bullet$};
\node(2) at (1.4,1.4){$\bullet$};
\node(3) at (2,0){$\bullet$};
\node(4) at (1.4,-1.4){$\bullet$};
\node(5) at (0,-2.1){$\cdots$};
\node(8) at (-1.4,1.4){$\bullet$};
\node(7) at (-2,0){$\bullet$};
\node(6) at (-1.4,-1.4){$\bullet$};
\node[above, yshift = 2pt] at (1){$v$};
\node[below] at (1twin){$v'$};
\node[left, xshift = -2pt] at (8){$u$};
\node[right, xshift = 2pt] at (2){$w$};
\end{tikzpicture}
\caption{$C_{n,v}$}
\end{subfigure}
\caption{Twinning of the path and cycle graphs.}
\label{fig: twin examples}

\end{figure}
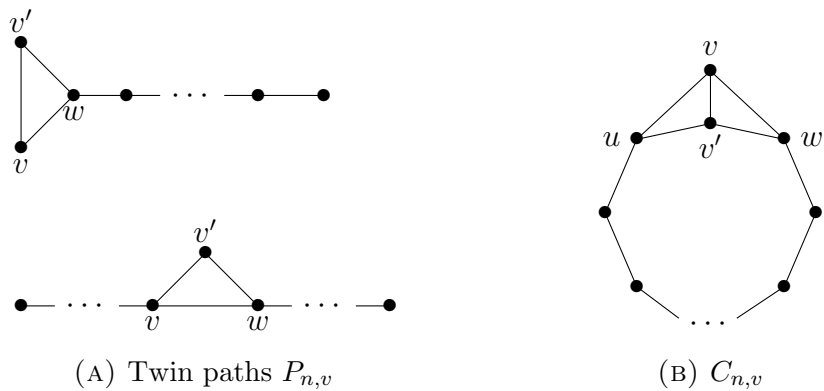

Twinning a non-isolated vertex always produces a triangle, so the triple deletion argument is a natural method to reduce the twinned graph back to the original one, as the following result shows.

\begin{cor}\label{leaf-twinS} 
Let $H$ be a graph on $n$ vertices and let $u$ be a vertex of $H$. Let $H'$ be the graph obtained by adding a new vertex $v$ and the edge $uv$ to $H$ and let $H''$ be the graph obtained by adding a new vertex $w$ and the edge $v w$ to $H'$. Finally let $H'_v$ be the graph $H'$ twinned at vertex $v$, with $v'$ denoting the new vertex. Then 
\[X_{H'_v}= 2(X_{H''}-e_2 X_H).\]
\end{cor}
    \begin{proof} This is clear by the triple deletion argument using the edges $uv, uv'$ of the triangle $\{u,v,v'\}$ as shown in~\Cref{fig:Cor_Leaf_TwinS}. Note also that $X_{P_2}=2e_2$.
\end{proof}

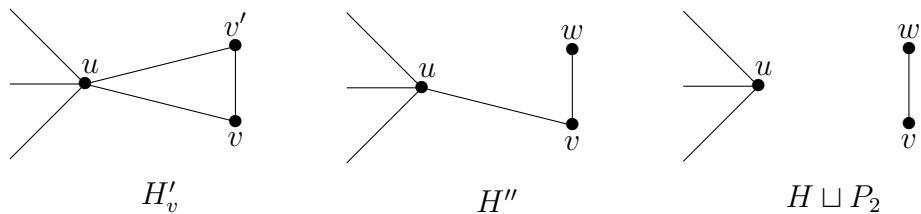
\begin{figure}[ht]
    \centering
\begin{tabular}{ccc}
\begin{tikzpicture}
\draw(0,0) -- (2,-0.5) -- (2,0.5) -- (0,0) -- (-1,-1);
\draw(-1,0) -- (0,0) -- (-1,1);
\node(u1) at (0,0){$\bullet$};
\node(v1) at (2,-0.5){$\bullet$};
\node(vTwin1) at (2,0.5){$\bullet$};
\node[above, xshift = 2pt] at (u1){$u$};
\node[below] at (v1){$v$};
\node[above] at (vTwin1){$v'$};
\node() at (1,-1.5){$H'_v$};
\end{tikzpicture} & \qquad 
\begin{tikzpicture}
\draw(0,0) -- (2,-0.5) -- (2,0.5);
\draw (0,0) -- (-1,-1);
\draw(-1,0) -- (0,0) -- (-1,1);
\node(u1) at (0,0){$\bullet$};
\node(v1) at (2,-0.5){$\bullet$};
\node(vTwin1) at (2,0.5){$\bullet$};
\node[above, xshift = 2pt] at (u1){$u$};
\node[below] at (v1){$v$};
\node[above] at (vTwin1){$w$};
\node() at (1,-1.5){$H''$};
\end{tikzpicture}&\qquad 
\begin{tikzpicture}
\draw(2,-0.5) -- (2,0.5);
\draw (0,0) -- (-1,-1);
\draw(-1,0) -- (0,0) -- (-1,1);
\node(u1) at (0,0){$\bullet$};
\node(v1) at (2,-0.5){$\bullet$};
\node(vTwin1) at (2,0.5){$\bullet$};
\node[above, xshift = 2pt] at (u1){$u$};
\node[below] at (v1){$v$};
\node[above] at (vTwin1){$w$};
\node() at (1,-1.5){$H \sqcup P_2$};
\end{tikzpicture}
\end{tabular}
    \caption{The triple deletion argument used in~\Cref{leaf-twinS} }
    \label{fig:Cor_Leaf_TwinS}
\end{figure}

\subsection{Known \texorpdfstring{$e$}{}-positivity results}\label{subsection: e-positivity background} 

Stanley  defined the chromatic symmetric function $X_G$ of a graph $G$ in 1995.  Since then, many families of graphs have  been examined. We provide an extensive, but by no means exhaustive, list of known $e$-positivity results in~\Cref{table:known_results}. We do not define these classes of graphs, but instead provide references containing their definitions as well as proofs of their $e$-positivity classification.  By convention, a family of graphs listed as ``not $e$-positive'' means that there is at least one graph in that class that is not $e$-positive. The table is roughly sorted chronologically by reference, and it is condensed so that some subclasses of other results are omitted.

\begin{table}[hb]
\begin{center}
\begin{tabular}{|| c | c | c||} 
 \hline
 Graph & Positivity & Reference   \\ [0.5ex] 
 \hline\hline
 Paths & $e$-positive & ~\cite[Proposition 5.3]{stanley_1995} \\ 
 \hline
 Cycles & $e$-positive  &~\cite[Proposition 5.4]{stanley_1995}  \\
 \hline
 Complete graphs & $e$-positive  &~\cite[Equation 3.1]{stanstem} \\
 \hline
 Co-triangle-free graphs & $e$-positive & ~\cite[Theorem 4.3]{stanstem} \\
 \hline 
 $K_\alpha$-chains & $e$-positive &~\cite[Corollary 7.7]{gebhard_sagan_2001} \\
 \hline
 Diamond and path chains & $e$-positive  &~\cite[Theorem 7.8]{gebhard_sagan_2001}\\
 \hline
 (claw, $P_4$)-free graphs & $e$-positive &~\cite[Theorem 1.4]{Tsujie_2018} \\
 \hline
  (claw, diamond)-free graphs & not $e$-positive &~\cite[Lemma 7]{hamel2017chromatic} \\
   \hline
   (claw, co-claw)-free graphs & not $e$-positive &~\cite[Lemma 7]{hamel2017chromatic} \\
   \hline
   (claw, $K_4$)-free graphs & not $e$-positive &~\cite[Lemma 7]{hamel2017chromatic} \\
   \hline
   (claw, $4K_1$)-free graphs & not $e$-positive &~\cite[Lemma 7]{hamel2017chromatic} \\
   \hline
   (claw, $2K_2$)-free graphs & not $e$-positive &~\cite[Lemma 7]{hamel2017chromatic} \\
   \hline
   (claw, $C_4$)-free graphs & not $e$-positive &~\cite[Lemma 7]{hamel2017chromatic} \\
   \hline
 (claw, paw)-free graphs & $e$-positive &~\cite[Theorem 3]{hamel2017chromatic} \\ 
 \hline
 (claw, co-paw)-free graphs & $e$-positive &~\cite[Theorem 4]{hamel2017chromatic} \\
 \hline
 Generalized bull graphs & $e$-positive &\cite[Theorem 3.7]{Cho-Huh_2017} \\
 \hline
 Lollipops and lariats &$e$-positive &~\cite[Theorem 9]{DahlbergVanWilligenburg2018}\\
 \hline
 $P_3$-free graphs & $e$-positive &~\cite[Theorem 5]{foley2018classes}\\
 \hline
 (claw, $K_3$)-free graphs & $e$-positive &~\cite[Theorem 5]{foley2018classes}\\
 \hline
 (claw, co-$P_3$)-free graphs & $e$-positive &~\cite[Theorem 5]{foley2018classes}\\
 \hline
 (co-claw)-free unit interval graphs & $e$-positive &\cite[Theorem 18]{foley2018classes} \\
 \hline
 Generalized pyramid graphs & $e$-positive &~\cite[Theorem 7]{li2019epositivity} \\
 \hline
 $2K_2$-free unit interval graphs & $e$-positive &~\cite[Theorem 13]{li2019epositivity} \\
 \hline
 Triangular ladders & $e$-positive &~\cite[Theorem 22]{dahlberg2019triangular}  \\
    \hline 
   Star graphs & not $e$-positive &~\cite[Example 11]{DahlbergSheVan2020} \\
   \hline
 Saltire and augmented saltire graphs & not $e$-positive &~\cite[Lemmas 4.4, 4.9]{dahlberg2020resolving} \\
 \hline
 Triangular tower graphs & not $e$-positive &~\cite[Lemma 5.4]{dahlberg2020resolving} \\
 \hline
 Tadpole graphs & $e$-positive &~\cite[Theorem 3.1]{Li-Li-Wang-Yang} \\
   \hline
   Line graphs of tadpole graphs & $e$-positive &~\cite[Corollary 3.3]{wang2021cyclechord} \\
   \hline
   Cycle-chord graphs  & $e$-positive &~\cite{wang2024}, \cite[Corollary 4.6]{wang2021cyclechord} \\
   \hline
   Kayak paddle graphs & $e$-positive &~\cite[Proposition 6.7]{aliniaeifard2021chromatic} \\
    \hline
 Generalized nets & not $e$-positive &\cite[Theorem 1]{foley2021spiders} \\
 \hline
   Melting $K_{\alpha} $-chains  & $e$-positive &~\cite[Theorem 3.19]{tom2023signed} \\
   \hline
\end{tabular}
\caption{Known $e$-positivity results and their references.}
\label{table:known_results}
\end{center}
\end{table}
\clearpage

\subsection{A new statistic on partitions}\label{subsec: epsilon-lambda}

In this subsection, we introduce a new statistic on the set of partitions that will allow us to describe $e$-coefficients more compactly. 

Recall that a \emph{partition} $\lambda$ of $n$, written $\lambda \vdash n$, is a weakly decreasing sequence of positive integers $\lambda = (\lambda_1,\ldots, \lambda_\ell)$ that sum to $n$, that is, $\sum_i \lambda_i = n$. We write $\ell = \ell(\lambda)$ for the \emph{length} of the partition, that is, the number of entries in the sequence. 
A partition $\lambda$ of $n$ can also be written   as $\lambda = \left\langle 1^{m_1}, 2^{m_2},\ldots, n^{m_n}\right\rangle$, where $m_i=m_i(\lambda)\geq 0$ denotes the multiplicity of the part $i$ in $\lambda$. 
The \emph{support} of $\lambda$, denoted $\supp(\lambda)$, is the set of distinct parts appearing in $\lambda$, that is,  $\supp(\lambda):=\{i\in \Z_{>0} : m_i(\lambda)\ge 1  \}$.  

Now we are ready to introduce the new statistic. 
\begin{definition}\label{def:support}
For a partition $\lambda$, define $\eps(\lambda)$ to be the quantity
\begin{equation}\label{eqn:eps-lambda} 
\eps(\lambda):= {\ell(\lambda)!\prod_{j\in \supp(\lambda)} \frac{(j-1)^{m_j(\lambda)}}{m_j(\lambda)!}} \qquad \text{ with } \qquad \eps(\emptyset)=1.
\end{equation}

Moreover, for a partition $\lambda$ of $n$ and a part $a$ such $m_a(\lambda)\ge 1$, let $\lambda-a$ denote the partition of $n-a$ obtained by deleting one part equal to $a$ from $\lambda$. By convention, if  $a$ is not a part of $\lambda$, we set $\eps(\lambda-a)=0$. 
\end{definition}
For example, $\eps((n))=n-1$ and $\eps((2^n))=1$ for any positive integer $n$. Additionally, $\epsilon(\lambda)=0$ if $\lambda$ contains a $1$. In~\Cref{table: epsilon}, we include several examples of partitions $\lambda$ together with their statistic $\eps(\lambda)$. 

\begin{table}[ht]
\begin{tabular}{l}
    \begin{tabular}{||c||c ||c ||c|c ||c|c ||c|c|c|c|} 
 \hline
$\lambda$ &(2) & (3) & (4) & (2,2) & (5) & (3,2)   & (6) & (4,2) & (3,3) & (2,2,2)\\ 
\hline 
$\eps(\lambda)$ &1& 2 & 3 & 1 & 4 & 4 & 5 & 6 & 4 &1\\
\hline
\end{tabular} \\[0.2in]
    \begin{tabular}{||c||c| c|c|c|| c|c|c|c|c|c|c|} 
 \hline
$\lambda$ &(7) & (5,2) & (4,3) & (3,2,2) & (8) & (6,2) & (5,3) & (4,4) & (4,2,2)& (3,3,2) & (2,2,2,2)\\ 
\hline 
$\eps(\lambda)$ &6 & 8 & 12 & 6 & 7 & 10 & 16 & 9 & 9 & 12 & 1\\
\hline
\end{tabular}
\end{tabular}
\caption{Examples of $\eps(\lambda)$ for some partitions $\lambda$.}
\label{table: epsilon}
\end{table}

\begin{remark}\label{remark:eps(lambda)-is-nonegative-integer}
     {Note that in~\eqref{eqn:eps-lambda}, 
\[\frac{\ell(\lambda)!}{m_j(\lambda)!}=\binom{\ell(\lambda)}{m_1(\lambda),\dots,m_n(\lambda)},\]}
which is always a nonnegative integer. Thus, $\eps(\lambda)$ is also a nonnegative integer. In fact, $\eps(\lambda)=0$ if and only if 1 is a part in $\lambda$, i.e., $m_1(\lambda)\ge 1$. 
\end{remark}

Next, we present other properties of $\eps(\lambda)$.

\begin{lemma}\label{lem:eps-propertieS} 
Let $\lambda$ and $\mu$ be partitions of $n$ and $m$, respectively. Then, we have the following:

\begin{enumerate}[(a)]
\itemsep0.12in
    \item For $j\in \supp(\lambda)$, $(j-1)\eps(\lambda-j)=m_j(\lambda)\dfrac{\eps(\lambda)}{\ell(\lambda)}$;
    \item $\eps(\lambda) =  \displaystyle{\sum_{j\in \supp(\lambda)}(j-1)\eps(\lambda-j)}
    $; and 
       \item $\displaystyle{\eps(\lambda)\eps(\mu)=\eps(\lambda\cup \mu) {\binom{\ell(\lambda\cup \mu)}{\ell(\lambda)}}^{-1}\prod_{j\in \supp(\lambda\cup \mu)}\binom{m_j(\lambda\cup\mu)}{m_j(\lambda)}}$ where $\lambda\cup \mu$ is the partition of $n+m$ formed by listing the parts of $\lambda$ and $\mu$ together in decreasing order.
\end{enumerate}
    
\end{lemma}
\begin{proof} 
\begin{enumerate}[(a)]
\itemsep0.1in
    \item Note first that both sides are identically zero if $1\in \supp(\lambda)$. For $j\in \supp(\lambda)$ with $j \ne 1$, this identity follows from the definition by noticing that 
    \[
    \eps(\lambda)=(j-1)\frac{\ell(\lambda)}{m_j(\lambda)}\left( (j-1)^{m_j(\lambda)-1}\prod_{l\in\supp(\lambda), l\ne j} (l-1)^{m_l(\lambda)} \frac{(\ell(\lambda)-1)!}{(m_j(\lambda-1)\prod_{l\ne j} m_l(\lambda)!}\right) .\]

    \item This identity follows from the definition of $\eps(\lambda)$, using $\displaystyle{\sum_{j\in\supp(\lambda)} m_j(\lambda)=\ell(\lambda)}$.

    \item This identity follows by expanding $\eps(\lambda\cup\mu)$,   using 
 $m_j(\lambda\cup\mu )=m_j(\lambda)+m_j(\mu)$ and
 $\ell(\lambda\cup\mu )=\ell(\lambda)+\ell(\mu)$.\qedhere
\end{enumerate}  
 \end{proof}
\begin{remark}
    Intuitively, the formula for $\eps(\lambda)$ can be interpreted as the number of pairs $(w,f)$ of words $w$ on the set $\{1,\ldots,\ell(\lambda)\}$ 
    of type $\lambda$, i.e. with $\lambda_i$ occurrences of the letter $i$,  
    together with a function $f:\{1,\ldots,\ell(\lambda)\}\to \Z$ satisfying $1\leq f(j)\leq \lambda_j-1$ for each $j\in \{1,\ldots,\ell(\lambda)\}$. These are exactly the \textit{codes} of Stembridge~\cite{stembridge_1992} with no fixed points and can be used to prove~\Cref{lem:eps-propertieS} 
    combinatorially. For example, the right-hand side of part $(b)$ can be interpreted as the number of ways of making a code of type $\lambda$ from a code whose type has length $\ell(\lambda)-1$. 
\end{remark} 

\section{\texorpdfstring{$e$}{e}-positivity via generating functions}\label{sec: generating functions}

For given family of graphs $G=\{G_n\}_{n\geq 0}$, one can show $e$-positivity of $X_{G_n}$ by showing that its generating function
\[\cX_G(z)=\sum_{n\geq 0}X_{G_n} z^n\]
can be written in the form
\begin{equation}\label{eq:rational-expr-CXG}\cX_{G}(z)=\frac{P(z)}{1-Q(z)},\end{equation}
where $P(z)$ and $Q(z)$ are $e$-positive formal power series in $z$. For the path $P_n$ and the cycle $C_n$, it is known from Stanley's original paper~\cite{stanley_1995} that this can be done. (See also~\cite[p.242]{CarScoVau1976} for paths.)

\begin{theorem}[{\cite[Propositions 5.3 and 5.4]{stanley_1995}}] \label{thm:RPS-gf-paths-cycles}
\begin{align*}
    \cX_P(z)&:=\sum_{n\geq 0} X_{P_n} z^n=\frac{\sum_{i\geq 0}e_i z^i}{1-\sum_{i\geq 1}(i-1)e_i z^i}, \\
    \cX_C(z) &:= \sum_{n\geq 2} X_{C_n} z^n=\frac{\sum_{i\geq 2}i(i-1)e_i  z^i}{1-\sum_{i\geq 1}(i- 1)e_i z^i}.
\end{align*}  
\end{theorem}

Note in particular that $X_{P_0}=1$.

In this section, we establish identities of the form~\eqref{eq:rational-expr-CXG}
for several families of twinned graphs 
by applying generating function techniques to the relations obtained from the triple deletion argument. 

It is useful to convert the preceding result to a recurrence relation for the chromatic symmetric function as follows.  We will use this formulation several times in this paper, notably in the proofs of~\Cref{lemma: fl formula} and~\Cref{theorem:twin-path-epoS-vertex}, as well as in~\Cref{sec:recurrences} .

\begin{prop}\label{prop:RECURSION-pathS-cycleS} 
We have the following recurrence relations:
\begin{enumerate}[(a)]
    \item\label{RECURSION-pathS} For $n\ge 3$, 
    $\displaystyle{X_{P_n}=n e_n +\sum_{j=2}^{n-1}(j-1) e_j X_{P_{n-j}}}
    =ne_n+\sum_{i=1}^{n-2} (n-i-1) e_{n-i} X_{P_i}$, 
    with initial conditions $X_{P_0}=1$, 
 $X_{P_1}=e_1$ and $ X_{P_2}=2 e_2$.

    \item\label{RECURSION-cycleS} For $n\ge 4$, $\displaystyle{X_{C_n}=n(n-1)e_n +\sum_{j=2}^{n-2}(j-1) e_j X_{C_{n-j}}}$, with initial conditions $X_{C_1}=0$, $X_{C_2}=2e_2$, and $X_{C_3}=6e_3$.
\end{enumerate}
\end{prop}

\subsection{Symmetric function identities and technical lemmas}
\label{sec:Sym-fn-identitieS}

In this section, we examine more closely the relationship between the generating function $E(z)$ for the elementary symmetric functions, and the generating functions $\cX_P(z)$ and $\cX_C(z)$ for the chromatic symmetric functions of the path and the cycle.
We also present some formulas for several families of coefficients appearing in the $e$-expansion of $\cX_P(z)$ and $\cX_C(z)$. 
We start by introducing some definitions and notation to facilitate our study. 

Let $E(z):=\sum_{i\ge 0} e_i z^i$ be the generating function for the elementary symmetric functions and define
\begin{equation*}
D(z):=E(z)-zE'(z)
=1-\sum_{i\ge 2}(i-1) e_i z^i.
\end{equation*}
\Cref{thm:RPS-gf-paths-cycles} can then be rewritten as: 
\begin{equation}\label{eqn:gf-pathS-cycleS}
\cX_P(z)=\frac{E(z)}{D(z)} \qquad \text{ and } \qquad \cX_C(z)=\frac{z^2 E''(z)}{D(z)}.
\end{equation}

It will be useful for our study to collect here the definitions of several $e$-positive series and their truncations and tails. Considering $k\geq 2$ whenever it appears, we define
\begin{equation}\label{eqn:KEY-epos-serieS}
\begin{array}{lcl}
&\qquad & \displaystyle{  E_{\ge k}(z)\!=\!\sum_{i\ge k}  e_i z^i,}  \\[0.2in]
\displaystyle{K(z)\!=\!\sum_{i\ge 2} i e_i z^i,}& \qquad  & \displaystyle{K_{\ge k}(z)\!=\!\sum_{i\ge k} i e_i z^i,}\\[0.2in] 
\displaystyle{G(z)\!=\!1\!-\!D(z)\!=\!\sum_{i\ge 2} (i-1) e_i z^i,} & & 
\displaystyle{G_{\ge k} (z)=\sum_{i\ge k} (i-1) e_i z^i,} \\[0.2in] 
\displaystyle{\frac{1}{D(z)}=\sum_{i\ge 0} G(z)^i}, &  & 
\displaystyle{ G_{\le k}(z)\!=\!\sum_{2\le i\le k} (i-1) e_i z^i\!=\!G(z)\!-\!G_{\ge k+1} (z)}.
 \end{array}
\end{equation}

The next lemma collects some $e$-positivity results concerning the generating functions introduced above. 
\begin{lemma}\label{lem:Ez-identitieS} 
\hspace{0.2cm}
\begin{enumerate}[(a)]
\itemsep0.2cm
\item The following expressions are $e$-positive: 
\begin{enumerate}[(i)]
\itemsep0.2cm
    \item $z^2 E''(z)-zE'(z)+e_1z$;
    \item $2 z^2 E''(z) -3z E'(z)+3e_1z+2e_2z^2$; and
    \item $z^2E''(z)-3zE'(z)+3E(z)+e_2z^2$.
\end{enumerate}

\item The following expressions can be written as rational functions with $e$-positive numerators: 
\begin{enumerate}[(i)]
\itemsep0.2cm
    \item $\cX_P(z)-(1+e_1z)$; and
    \item $(1+e_1z) \cX_C(z) -\cX_P(z) +1+e_1z$.
\end{enumerate}

\end{enumerate}
\end{lemma}

\begin{proof} \phantom{nothing}
\begin{enumerate}[(a)]
    \itemsep0.1in
\item The $e$-positivity results follow, respectively, from the identities:
\begin{enumerate}[(i)]
\itemsep0.3cm
    \item 
    $z^2 E''(z)-zE'(z)=-e_1z+\sum_{i\ge 3} i(i-2) e_i z^i$;
    \item 
    $2 z^2 E''(z) -3z E'(z)= -3e_1z-2e_2z^2 + \sum_{i\ge 3} (2i^2-5i) e_i z^i$; and
    \item  
    $z^2E''(z)-3zE'(z)+3E(z)=3-e_2z^2+\sum_{i\ge 4} (i-1)(i-3) e_i z^i$.
\end{enumerate}

\item For the results concerning $e$-positive numerators, we have that:
\begin{enumerate}[(i)]
    \item By~\eqref{eqn:gf-pathS-cycleS},
\begin{align*}
\cX_P(z)-(1+e_1z)
&=\frac{zE'(z)+e_1z[zE'(z)-E(z)]}{D(z)} \\
&=\frac{\sum_{i\ge 2} i e_i z^i +e_1z\sum_{i\ge 2} (i-1) e_i z^i }
{D(z)}.
\end{align*}

\item By the previous item,
\begin{align*}
(1+e_1z) \cX_C(z) -\cX_P(z) +1+e_1z  
&=\frac{(1+e_1z)[z^2E''(z)-zE'(z)] +e_1z E(z)}{D(z)}\\
&=\frac{(1+e_1z)F_1(z) +e_1z(E(z)-1-e_1z)}{D(z)},
\end{align*} 
where $F_1(z) =\sum_{i\ge 3} i(i-2) e_i z^i$ and 
$E(z)-1-e_1z=\sum_{i\ge2} e_iz^i$ are $e$-positive. \qedhere
\end{enumerate}
\end{enumerate}
\end{proof}

Using~\Cref{lem:Ez-identitieS}, we get an $e$-positive expression for the generating function for paths that isolates those terms containing $e_1$ and that is different from the one in~\Cref{thm:RPS-gf-paths-cycles}. 
\begin{cor}\label{cor:new-gf-PATHSS}
We have 
\begin{equation*}
\cX_P(z)=\frac{K(z)}{D(z)} +e_1z\frac{G(z)}{D(z)} +(1+e_1z).
    \end{equation*}
\end{cor}

\vspace{0.5cm}

We now analyze these generating functions to extract closed formulas for the coefficients in the $e$-expansions. Recall the statistic on partitions $\varepsilon(\lambda)$ introduced in~\Cref{subsec: epsilon-lambda}. 

We start with a result that shows the relation between the coefficients of $G(z)^k$ and $\dfrac{1}{D(z)}$ in their $e$-expansion and $\eps(\lambda)$.
\begin{lemma}\label{lem:expandS-1overDz}
The coefficient of $e_\lambda z^{|\lambda|}$ in $G(z)^k$ is $\eps(\lambda)$ and hence 
\[\frac{1}{D(z)}=\sum_{\lambda} \eps(\lambda)e_\lambda z^{|\lambda|},\]
where the sum is over all partitions $\lambda$.
\end{lemma}

\begin{proof}
    This follows by manipulating the formal series directly:
    \[
        G(z)^k=\lp \sum_{i\geq 2} (i-1)e_iz^i\rp^k
        =\sum_{{\substack{\lambda\\ \ell(\lambda)=k}}} e_\lambda  z^{|\lambda|}\prod_{i\geq 2} (i-1)^{m_i(\lambda)}
        =\sum_{{\substack{\lambda\\ \ell(\lambda)=k}}} \eps(\lambda) e_\lambda   z^{|\lambda|}. \qedhere
    \]
\end{proof}

We end this subsection by showing that several families of coefficients in the $e$-expansion of $\cX_P(z)$ and $\cX_C(z)$   can be expressed compactly in terms of $\eps(\lambda)$. (See also~\cite{Wolf_CSF_Path_Cycle}.)

\begin{prop}\label{prop:coeffs-pathS-cycleS}
Given a graph $G$, let $c_\lambda$ be the coefficient of $z^{|\lambda|}e_\lambda$ in $\mathcal{X}_{G}$, that is, $\mathcal{X}_{G}=\sum c_\lambda z^{|\lambda|} e_{\lambda}$. Then, we have the following:
\begin{enumerate}[(a)]
\itemsep0.1in
\item For $G=P_n$, 
$c_\lambda = \displaystyle{\eps(\lambda)+\sum_{a\in \supp(\lambda)} \eps(\lambda-a)
=\sum_{a\in\supp(\lambda)} a\,\eps(\lambda-a)}$. 

In particular, if $\lambda = 1 \cup \mu$ for some partition $\mu$, then $\displaystyle{c_\lambda = \sum_{\substack{a\in \supp(\mu)\\ a\ge 2}} (a-1) \eps(\mu-a)}$.

Moreover, we can also extract  particular coefficients like 
\begin{equation*}
c_{(n)}=n,
 \qquad  c_{(n-1,1)}=n-2,  \qquad  c_{(2^k)}=2,  \quad \text{and} \quad c_{(2^k,1)}=1.
\end{equation*}
\item For $G=C_n$, we have that

$c_\lambda = \displaystyle{\sum_{a\in \supp(\lambda)} a(a-1) \,\eps(\lambda-a)}$.
\end{enumerate}
\end{prop}
\begin{proof} We use the generating functions in~\eqref{eqn:gf-pathS-cycleS}.
\begin{enumerate}[(a)]

    \item The first expression comes directly from the path generating function  $\cX_P(z)$  and the second expression also follows from~\Cref{lem:Ez-identitieS}. The equality of the two expressions and the case when  $\lambda = 1\cup \mu$ follow using~\Cref{lem:eps-propertieS}. 

    \item The formula for this coefficient comes directly from the cycle generating function $\cX_C(z)$. 
    \endproof
\end{enumerate}
\end{proof}
%


%

\subsection{Generating functions for twinned paths}
In this section, we focus on studying the various ways to twin a path. The following is a key result. 
\begin{lemma}\label{lem:KEY-epos-special-fractionS} 
For $k\ge 2$, the rational function $\dfrac{1-G_{\le k}(z)}{D(z)}$ and the function $\cX_P(z)(1-G_{\le k}(z))$ are $e$-positive.
\end{lemma}
\begin{proof} 
For the rational function, we have that
\begin{equation}\label{eq:quotientGD}
\frac{1-G_{\le k}(z)}{D(z)}=\frac{1-G(z)+G_{\ge k+1}(z)}{1-G(z)}=1+\frac{G_{\ge k+1}(z)}{D(z)}.
\end{equation}

This is $e$-positive since $\displaystyle{G_{\ge k+1}(z)=\sum_{i\ge k+1} (i-1)e_iz^i}$ and $\dfrac{1}{D(z)}$ expands $e$-positively in powers of $G(z)$. 
The $e$-positivity of the second function follows from~\eqref{eqn:gf-pathS-cycleS} and the identity
\[\cX_P(z)\left(1-G_{\le k}(z)\right)= E(z) + \cX_P(z)\, G_{\ge k+1}(z). \qedhere\]
\end{proof}

\subsubsection{Paths twinned at a leaf} 
\phantom{}

The recurrence for the chromatic symmetric function of twinned paths at a leaf (i.e., a vertex of degree 1)  appears in Dahlberg and van Willigenburg~\cite[Equation 5]{DahlbergVanWilligenburg2018}, where the graph  $P_{n,v}$ with $v$ a leaf is called the \emph{lariat graph} $L_{n+3}$. Its chromatic symmetric function had been considered earlier by Wolfe in~\cite{Wolf_CSF_Path_Cycle}, and $e$-positivity was first established by Gebhard and Sagan in~\cite[Corollary 7.7]{gebhard_sagan_2001}.

\begin{prop}\label{prop:gf-twin-path-end-vertex}  Let $v$ be a leaf of the path $P_n$. The generating function for the chromatic symmetric function of the twin $P_{n,v}$ of a path on $n$ vertices satisfies the following identity:
\[2 +2 e_1z+\sum_{n\ge 1} X_{P_{n,v}} z^{n+1} =2(1-e_2z^2)  \cX_P(z).\]
\end{prop}
\begin{proof}
By~\cite[Equation 5]{DahlbergVanWilligenburg2018}, 
 the chromatic symmetric function of $P_{n,v}$, with $n\ge 1$, is given by 
    \begin{equation}\label{eqn:twin-path-rec}
    X_{P_{n,v}} = 2 X_{P_{n+1}} - X_{P_{2}}X_{P_{n-1}}.
    \end{equation}
The proof now follows by using the generating function $\cX_P(z)$.
\end{proof}

Now we are ready to derive a  generating function for paths twinned at a leaf. Although the $e$-positivity was established in~\cite[Corollary 7.7]{gebhard_sagan_2001} and again in~\cite{DahlbergVanWilligenburg2018}, as mentioned earlier, our contribution here is to give the manifestly $e$-positive generating function below for $X_{P_{n,v}}$, using only symmetric functions, which enables a more efficient coefficient extraction.
\begin{prop}\label{prop:epos-via-gf-twin-pathS-leaf} 
Let $\cX_{P_v}(z)$ be the generating function for the twinned path at a leaf, that is, $\cX_{P_v} :=\sum_{n\ge 1} X_{P_{n,v}} z^{n+1}$. Then
\[\frac{1}{2} \cX_{P_v} (z)=K(z)\frac{G_{\ge 3}(z)}{D(z)} +e_1z G(z)\frac{G_{\ge 3}(z)}{D(z)} +e_2z^2 +\sum_{i\ge 3} i e_i z^i +e_1z G_{\ge 3}(z).
\]
In particular $X_{P_{n,v}}$ is $e$-positive, and the initial values are
\[X_{P_{1,v}}=2e_2, \quad X_{P_{2,v}}=2(3e_3), \quad X_{P_{3,v}}=2(4e_4+2e_1e_3), \quad X_{P_{4,v}}=2(4e_2e_3+3e_1e_4+5e_5).\]

An $e$-positive expression without denominators in terms of the path generating function $\cX_P$ is
\[\frac{1}{2} \cX_{P_v} (z)=\cX_P(z) G_{\ge 3}(z) +\sum_{i\ge 2} e_i z^i.
\]
\end{prop}
\begin{proof} 
    By~\Cref{lem:KEY-epos-special-fractionS} and ~\Cref{cor:new-gf-PATHSS}, we have  that
\begin{align*}
    \frac{1}{2} \cX_{P_v}(z) &= (1-e_2z^2 ) \cX_P -(1+e_1z)
    =K(z) \frac{1-e_2z^2 }{D(z)} 
        +e_1z G(z)\frac{1-e_2z^2 }{D(z)} -e_2z^2(1+e_1z)\\
    &=(K(z)+e_1zG(z))\frac{G_{\ge 3}(z)}{D(z)} + (K(z)+e_1zG(z)) -e_2z^2(1+e_1z)\\
    &=(K(z)+e_1zG(z))\frac{G_{\ge 3}(z)}{D(z)} +(K(z)-e_2z^2 ) +e_1z (G(z)-e_2z^2 ).
\end{align*}
Since $K(z)-e_2z^2=e_2z^2  +\sum_{i\ge 3} i e_i z^i$ and $G(z)-e_2z^2 =\sum_{i\ge 3} (i-1) e_i z^i$, the result follows.

The second expression is obtained from the first by rewriting the formula in \Cref{cor:new-gf-PATHSS} as follows:
\[\cX_P(z)= \frac{K(z)+e_1z G(z)}{D(z)} +(1+e_1z). \qedhere\]
\end{proof}

\begin{cor}\label{cor:e-coefficents-of-twinned-path}
Let $c_\lambda$ be the coefficient of $e_\lambda z^{|\lambda|} $ in $\mathcal{X}_{P_{v}}$, that is $\mathcal{X}_{P_{v}}=\sum c_\lambda e_{\lambda} z^{|\lambda|} $, where $v$ is a leaf of the path $P_n$.
The following is a list of closed formulas for all the coefficients $c_\lambda$ involved in the general expression of $\cX_{P_v}(z)$:
\begin{enumerate}[(a)]
\itemsep0.1in
    \item $c_{(k)}= 2k$, $k\ge 3$, and $c_{(2)}=2$;
    \item $c_{(k-1,1)}=2(k-2)$, $k\ge 4 $;
    \item $c_{(k-2,2)}=4(k-3)$, $k\ge 5$; 
    \item $c_{(i,j)}=2i(j-1)+2j(i-1)=2(2ij-i-j)$, $i>j\ge 3$;
    \item $c_{(i,i)}=2i(i-1),$ $i\ge 3$.
    \item $c_{(3, 2^{k})}=8$ and $c_{(3, 2^{k},1)}=4$, $k\ge 2$. %
    \item If $c_{1\cup\mu}\ne 0$ and $\ell(\mu)\ge 2$, then $1 \notin \supp(\mu)$ and there exists $a\geq 3$ such that $a\in \supp(\mu)$. In particular  $c_{(2^{k},1)}=0$.  The coefficient of $e_{1\cup\mu}$ is equal to twice the coefficient of $e_\mu$ in 
    $G(z)G_{\ge 3}(z)\, G(z)^{\ell(\mu)-2}$, 
    and it equals 
    \[2\sum_{\substack{(a,b)\\a,b\in\supp(\mu)\\ a\ge 2, b\ge 3}} (a-1)(b-1) \eps((\mu-a)-b) .\]
    \item Assume $1\notin \supp(\lambda)$ and $\ell(\lambda)\ge 2$.  If $c_\lambda\ne 0$, then $\lambda$ contains at least one part of size at least 3.
    In particular  $c_{(2^{k})}=0$. 
    The coefficient $c_\lambda$ is equal to  twice the coefficient of $e_{\lambda}$ in $K(z)G_{\ge 3}(z)\,G(z)^{\ell(\lambda)-2}$, and it equals
    \[2\sum_{\substack{(a,b)\\a,b\in\supp(\lambda)\\ a\ge 2, b\ge 3}} a(b-1) \eps((\lambda-a)-b).\]
\end{enumerate}
\end{cor}

Note that cases \textit{(c)-(f)} are particular cases of \textit{(g)} and \textit{(h)}. 

\subsubsection{Paths twinned at both leaves}
\phantom{}

In this section, we consider the twinned path $P_{n,w,v}$ at both leaves, which we label with $w$ and $v$. The $e$-positivity of its chromatic symmetric function is a consequence of ~\cite[Corollary 7.7]{gebhard_sagan_2001}, whose proof relies on the theory of symmetric functions in noncommutating variables.  Here we derive an $e$-positive generating function using only symmetric function identities.

Unlike the other families of graphs, here one needs to pay special attention to the smaller values of $n$. 
We consider the special case of the path on two vertices first.
Twinning both vertices produces 
the twin of the cycle graph $C_3$ at one vertex, 
which is also the complete graph $K_4$, as shown in~\Cref{fig: special cases}, and therefore  
we have the following.
\begin{lemma}\label{lem:P2-double-twinS} 
For the path $P_2$ twinned at both vertices, 
$X_{P_{2,v,w}}=X_{C_{3,v}}=24 e_4.$
\end{lemma}

\begin{figure}[ht]
    \centering
\begin{tikzpicture}
\draw(1,0) -- (-1,0) -- (0,1) -- (1,0);
\node(w) at (1,0){$\bullet$};
\node(v) at (-1,0){$\bullet$};
\node(v') at (0,1){$\bullet$};
\node[left] at (v){$v$};
\node[right] at (w){$w$};
\node[above] at (v'){$v'$};
\node at (0,-2.25) {$P_{2,v}$};
\draw(5,0) -- (4,1) -- (3,0) -- (4,-1) -- (5,0);
\draw(5,0) -- (3,0);
\draw(4,1) -- (4,-1);
\node(w1) at (5,0){$\bullet$};
\node(v1) at (3,0){$\bullet$};
\node(v1') at (4,1){$\bullet$};
\node(w1') at (4,-1){$\bullet$};
\node[left] at (v1){$v$};
\node[right] at (w1){$w$};
\node[above] at (v1'){$v'$};
\node[below] at (w1'){$w'$};
\node at (4,-2.25) {$P_{2,v,w}$};
\draw(9,0) -- (8,1) -- (7,0) -- (8,-1) -- (9,0);
\draw(9,0) -- (7,0);
\draw(8,1) -- (8,-1);
\node(w2) at (9,0){$\bullet$};
\node(v2) at (7,0){$\bullet$};
\node(v2') at (8,1){$\bullet$};
\node(u) at (8,-1){$\bullet$};
\node[left] at (v2){$v$};
\node[right] at (w2){$w$};
\node[above] at (v2'){$v'$};
\node[below] at (u){$u$};
\node at (8,-2.25) {$C_{3,v}$};
\end{tikzpicture}
    \caption{}
    \label{fig: special cases}
\end{figure}

For the general case, we start with a consequence of the triple deletion argument. 

\begin{cor}\label{prop:Twinned-both-leaves-pathS}
Let $v,w$ be the two leaves of the path $P_n$, and let $P_{n,v,w}$ be the path twinned at both leaves. Then, for $n\geq 3$, 
\begin{equation}
    X_{P_{n,v,w}} =2 X_{P_{n+1,v}}-2 e_2 X_{P_{n-1,v}}=4(X_{P_{n+2}}-2e_2X_{P_n}+e_2^2 X_{P_{n-2}}).
\end{equation}
\end{cor}

This  relation allows us to give the following generating function identity.
\begin{prop} For the graph $P_{n,v,w}$
, we have 
\begin{equation}\label{eqn:gf-double-twinS}
\frac{1}{4}\sum_{n\ge 3} X_{P_{n,v,w}} z^{n+2}
=(1-e_2z^2)\frac{1}{2}\cX_{P_{v}}+\frac{1}{2}\alpha,
\end{equation}
where $\alpha=2e_2^2 z^4-(8 e_4 z^4 +4e_3e_1z^4 +6 e_3 z^3  +2e_2z^2)$.
\end{prop}
\begin{proof}
    Multiply both sides of the first equality in~\Cref{prop:Twinned-both-leaves-pathS} %
    by $z^{n+2}$ and sum for $n\geq 3$:
    \begin{align*}
        \sum_{n\geq 3}X_{P_{n,v,w}}z^{n+2} &=2\sum_{n\geq 3} X_{P_{n+1,v}}z^{n+2}-2 e_2 \sum_{n\geq 3}X_{P_{n-1,v}}z^{n+2}\\
        &=2\sum_{n\geq 4} X_{P_{n,v}}z^{n+1}-2 e_2 z^2\sum_{n\geq 2}X_{P_{n,v}}z^{n+1}\\
        &=2(1-e_2z^2)\cX_{P_v} - 2(X_{P_{3,v}}z^4+X_{P_{2,v}}z^3+X_{P_{1,v}}z^2-2e_2z^2 X_{P_{1,v}}z^2)\\
        &=2(1-e_2z^2)\cX_{P_v} - 2[(8e_4+4e_3e_1)z^4+6e_3z^3+2e_2z^2-2e_2^2z^4]
    \end{align*}
    where the computations for $X_{P_{n,v}}$ follow from~\Cref{prop:epos-via-gf-twin-pathS-leaf}.
\end{proof}

The next theorem follows from~\Cref{prop:Twinned-both-leaves-pathS} and manipulation of the formal series.%
\begin{theorem}\label{thm:gf-double-twinS-deg-ge7-len3}
    The generating function $\displaystyle{\frac{1}{4}\sum_{n\ge 3} X_{P_{n,v,w}} z^{n+2}}$ has the following $e$-positive expansion:
\begin{align*}
   \frac{1}{4}\sum_{n\ge 3} X_{P_{n,v,w}}z^{n+2} 
   &=\left(K(z)+e_1zG(z)   \right)\frac{G_{\ge3}(z)^2}{D(z)}+e_1z G_{\ge3}(z)^2 \\ 
   &\quad+\left(G_{\ge3}(z)\sum_{i\ge 3} ie_iz^i + e_1z \sum_{i\ge 4}(i-1) e_i z^i+e_2z^2 \sum_{i\ge 3}(i-2) e_i z^i\right) +\sum_{i\ge 5}i e_i z^i \nonumber.
\end{align*}

An $e$-positive expression without denominators in terms of the path generating function $\cX_P$ is 
\[\cX_P(z) G_{\ge 3}^2(z) + K_{\ge 5}(z) 
+ G_{\ge 3}(z) \sum_{i\ge 3}e_iz^i + e_1 z G_{\ge 4}(z) +e_2 z^2 \sum_{i\ge 3} (i-2) e_i z^i.
\]

\end{theorem}

\begin{proof}
We use the generating function in~\Cref{prop:epos-via-gf-twin-pathS-leaf} to expand $\frac{1}{2} (1-e_2z^2) \cX_{P_{n,v}}$ as 
\begin{align*} 
\frac{1}{2} (1-e_2z^2) \cX_{P_{v}}
=&(1-e_2z^2) \left(K(z)+z e_1 G(z)\right) \frac{G_{\ge 3}(z)}{D(z)} \\
&\quad+(1-e_2z^2)\left(e_2z^2 +\sum_{i\ge 3} i e_i z^i +z e_1G_{\ge 3}(z)\right).
\end{align*}

By~\eqref{eq:quotientGD}, $\dfrac{1-e_2z^2}{D(z)} = 1+ \dfrac{G_{\geq 3}(z)}{D(z)}$, and we can rewrite the above expression as
\begin{align*}
\frac{1}{2} (1-e_2z^2) \cX_{P_{v}}
&= \left(K(z)+z e_1 G(z)\right)  G_{\ge 3}(z) \left(1+ \dfrac{G_{\geq 3}(z)}{D(z)}\right)\\ 
&\quad+(1-e_2z^2)\left(e_2z^2 +\sum_{i\ge 3} i e_i z^i +z e_1G_{\ge 3}(z)\right).
\end{align*}

Next, we arrange the expression above so that the term $-\frac{1}{2}\alpha$ appears:
\begin{align*}
\frac{1}{2} (1-e_2z^2) \cX_{P_{v}} 
&=\left(K(z)+z e_1 G(z)\right)\frac{G_{\ge 3}(z)^2}{D(z)}
-\frac{1}{2}\alpha +\sum_{i\ge 5} ie_i z^i+e_1z \sum_{i\ge 4} (i-1) e_i z^i\\
&\quad+\left(K(z)+z e_1 G(z)\right)G_{\ge 3}(z)-e_2z^2  \sum_{i\ge 3} i e_i z^i- e_2z^2  (z e_1) G_{\ge 3}(z).
\end{align*}

Thus, we have that 
\begin{align}
    \frac{1}{4}\sum_{n\ge 3} X_{P_{n,v,w}} z^{n+2}
&=\frac{1}{2}(1-e_2z^2)\cX_{P_{v}}+\frac{1}{2}\alpha \nonumber \\
&=\left(K(z)+z e_1 G(z)\right)\frac{G_{\ge 3}(z)^2}{D(z)}
+\sum_{i\ge 5} ie_i z^i+e_1z \sum_{i\ge 4} (i-1) e_i z^i \label{eqn:manupulation-twinning-leaves-1} \\
&\quad+\left(K(z)+z e_1 G(z)\right)G_{\ge 3}(z)-e_2z^2  \sum_{i\ge 3} i e_i z^i- e_2z^2  (z e_1) G_{\ge 3}(z),\label{eqn:manupulation-twinning-leaves-2}
\end{align}
where the terms in line~\eqref{eqn:manupulation-twinning-leaves-1} are $e$-positive. 
Thus, we only need to show that the terms in line~\eqref{eqn:manupulation-twinning-leaves-2} are also $e$-positive. 
Note that 
\[K(z)G_{\ge 3}(z)=\left(2e_2z^2+\sum_{i\geq 3}ie_iz^i\right)G_{\geq 3}(z)=2e_2z^2G_{\ge 3}(z) + G_{\ge 3}(z)\sum_{i\ge 3}i e_i z^i.\] 
Together with the fact that $G(z)-e_2z^2 =G_{\ge 3}(z)$, line~\eqref{eqn:manupulation-twinning-leaves-2} can be written as
\begin{align}
&\left(K(z)+z e_1 G(z)\right)G_{\ge 3}(z)-e_2z^2  \sum_{i\ge 3} i e_i z^i- e_2z^2  (z e_1) G_{\ge 3}(z) \nonumber\\
&=G_{\ge 3}(z)\sum_{i\ge 3}i e_i z^i +e_2z^2 \left(2\sum_{i\ge 3} (i-1) e_i z^i-\sum_{i\ge 3} i e_i z^i \right)+ e_1z (G(z)-e_2z^2 ) G_{\ge 3}(z) \nonumber\\
&=G_{\ge 3}(z)\sum_{i\ge 3}i e_i z^i +e_2z^2 \sum_{i\ge 3} (i-2) e_i z^i+ e_1z  G_{\ge 3}(z)^2. \label{eqn:manupulation-twinning-leaves-3}
\end{align}
Since the expression in~\eqref{eqn:manupulation-twinning-leaves-3} is also $e$-positive, the result follows. 

The second expression involving $\cX_P$ follows as in the proof of \Cref{prop:epos-via-gf-twin-pathS-leaf}.
\end{proof}

In particular, we can extract the following  formulas for the  coefficients.
\begin{cor}\label{cor:double-twin-path-leaveS-coeffS}
Let $c_\lambda$ be the coefficient of $e_\lambda z^{|\lambda|}$ in $\mathcal{X}_{P_{v,w}}$, that is, $\mathcal{X}_{P_{v,w}}=\sum c_\lambda e^{\lambda}z^{|\lambda|}$. We have the following list of closed formulas:
\begin{enumerate}[(a)] 
\itemsep0.1in
\item For $k\geq 3$, $c_{(k+2)}=4(k+2)$,   $c_{(k,2)}=4(k-2)$, and $c_{(k+1,1)}=4k$. 
\item For $i\ge 3$, $c_{(i,i)}=4(i-1)i$, and for $i,j\ge 3, i\ne j$, $c_{(i,j)}=4(j-1)i+4(i-1)j$.
\item For $i,j\geq 3$, $i\neq j$,  
$c_{(i,i,1)}=4(i-1)^2$,  $c_{(i,j,1)}=8(i-1)(j-1)$, $i,j\ge 3, i\ne j$, and zero otherwise.
\item If $c_{1\cup\mu}\ne 0$, then $1\notin \supp(\mu)$.
\item For all $k\ge 0$, $c_{(3^2,2^{k+1})}=32 $ and $c_{(3^2,2^{k+1},1)}= 16$.
\end{enumerate}
\end{cor}

\subsubsection{Paths twinned at an interior vertex}
\phantom{}

In this section, we establish an $e$-positive generating function for the path $P_{n,\ell}$ twinned at an interior vertex $\ell$, where we label the vertices of $P_n$ by $1,\,2,\ldots, n$ from left to right. As stated in the introduction, the $e$-positivity  can also be deduced from ~\cite[Theorem 7.8]{gebhard_sagan_2001}.

As in the preceding section, we first derive a triple deletion formula for the chromatic symmetric function of $P_{n,\ell}$, 
(\Cref{JM}), and then deduce an $e$-positive generating function for its chromatic symmetric function (\Cref{thm:cX_P-f_ell-formula}).
We begin with some definitions. 
\begin{definition}
    For $n\geq 2$ and $1\leq \ell \leq n-1$, let $\tilde{T}_{n, \ell}$ (\textit{T} for \textit{triangle}) denote the graph obtained from the path graph $P_n$ by adding a vertex adjacent to both $\ell$ and $\ell+1$.  
    For $n\geq 1$ and $1\leq \ell \leq n$, let $F_{n,\ell}$ (\textit{F} for \textit{flagpole}) denote the graph obtained from $P_n$ by adding a vertex adjacent to $\ell$. 
\end{definition}

By the triple deletion argument illustrated in~\Cref{firstpic}, we have the following result.

\begin{lemma}\label{firststep} 
For $n\geq 3$ and $2\leq \ell\leq n-1$, we have
    \[X_{P_{n,\ell}} = 2X_{\tilde{T}_{n,\ell -1}} - X_{\tilde{T}_{\ell, \ell-1}}  X_{P_{n-\ell}}.\] 
\end{lemma}

\begin{figure}[hb]
    \centering
\begin{tikzpicture}[scale = 1.2]
\node(dots1) at (1.25,0){$\cdots$};
\node(1) at (0,0){$\bullet$};
\node(l-1) at (2.5,0){$\bullet$};
\node(l) at (3.5,0){$\bullet$};
\node(twin) at (3.5,1){$\bullet$};
\node(l+1) at (4.5,0){$\bullet$};
\node(dots2) at (5.75,0){$\cdots$};
\node(n) at (7,0){$\bullet$};
\draw(0,0) -- (dots1);
\draw (dots1) -- (dots2);
\draw(2.5,0) -- (3.5,1) -- (3.5,0);
\draw(3.5,1) -- (4.5,0);
\draw(dots2) -- (7,0);
\node[below, scale = 0.8] at (1){$1$};
\node[below, scale = 0.8, xshift = -3pt] at (l-1){$\ell-1$};
\node[below, scale = 0.8, xshift = 3pt] at (l){$\ell$};
\node[above, xshift = 5pt, scale = 0.8] at (twin){$\ell'$};
\node[below,scale = 0.8, xshift = -2pt] at (l+1){$\ell+1$};
\node[right] at (3.8,0.7){$\epsilon_1$};
\node[above, yshift = -3pt] at (4,0){$\epsilon_2$};
\node[left, xshift = 3pt] at (3.5,0.4){$\epsilon_3$};
\node[below] at (n){$n$};
\node[] at (1,0.75){$P_{n,\ell}$};
\end{tikzpicture}
\begin{tikzpicture}[scale = 1.2]
\node(dots1) at (1.25,0){$\cdots$};
\node(1) at (0,0){$\bullet$};
\node(l-1) at (2.5,0){$\bullet$};
\node(l) at (3.5,0){$\bullet$};
\node(twin) at (3.5,1){$\bullet$};
\node(l+1) at (4.5,0){$\bullet$};
\node(dots2) at (5.75,0){$\cdots$};
\node(n) at (7,0){$\bullet$};
\draw(0,0) -- (dots1);
\draw (dots1) -- (dots2);
\draw(2.5,0) -- (3.5,1) -- (3.5,0);
\draw(dots2) -- (7,0);
\node[below, scale = 0.8] at (1){$1$};
\node[below, scale = 0.8, xshift = -3pt] at (l-1){$\ell-1$};
\node[below, scale = 0.8, xshift = 3pt] at (l){$\ell$};
\node[above, xshift = 5pt, scale = 0.8] at (twin){$\ell'$};
\node[below,scale = 0.8, xshift = -2pt] at (l+1){$\ell+1$};
\node[below] at (n){$n$};
\node[] at (1,0.75){$\widetilde{T}_{n,\ell-1}$};
\end{tikzpicture}
\begin{tikzpicture}[scale = 1.2]
\node(dots1) at (1.25,0){$\cdots$};
\node(1) at (0,0){$\bullet$};
\node(l-1) at (2.5,0){$\bullet$};
\node(l) at (3.5,0){$\bullet$};
\node(twin) at (3.5,1){$\bullet$};
\node(l+1) at (4.5,0){$\bullet$};
\node(dots2) at (5.75,0){$\cdots$};
\node(n) at (7,0){$\bullet$};
\draw(0,0) -- (dots1);
\draw (dots1) -- (l);
\draw(l+1) -- (dots2);
\draw(2.5,0) -- (3.5,1) -- (3.5,0);
\draw(dots2) -- (7,0);
\node[below, scale = 0.8] at (1){$1$};
\node[below, scale = 0.8, xshift = -3pt] at (l-1){$\ell-1$};
\node[below, scale = 0.8, xshift = 3pt] at (l){$\ell$};
\node[above, xshift = 5pt, scale = 0.8] at (twin){$\ell'$};
\node[below,scale = 0.8, xshift = -2pt] at (l+1){$\ell+1$};
\node[below] at (n){$n$};
\node[] at (1,0.75){$\widetilde{T}_{\ell,\ell-1} \sqcup P_{n-\ell}$};
\end{tikzpicture}
    \caption{The triple deletion argument applied as in~\Cref{firststep}.}
    \label{firstpic}
\end{figure}
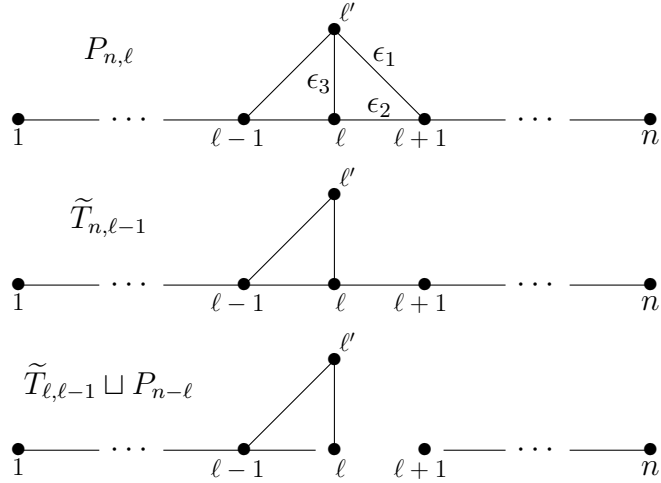

By carefully applying the triple deletion argument to various $P_{n,\ell}$, we can deal with the triangles $\widetilde{T}_{n,\ell}$ by ``shifting" them around. Note that $\widetilde{T}_{n, \ell}$ has a triangle, and so the triple deletion argument applies to two different sets of edges, to which we refer as \emph{left} and \emph{right shifts}. We illustrate them on the left-hand side and right-hand side of~\Cref{fig:LeftAndRight}, respectively. 

\begin{figure}[ht]
\centering
\begin{tabular}{cc}
\begin{tikzpicture}
\node(dots1) at (1.25,0){$\cdots$};
\node(1) at (0,0){$\bullet$};
\node(l-k-1) at (2.5,0){$\bullet$};
\node(l-k) at (3.5,0){$\bullet$};
\node(twin) at (3.5,1){$\bullet$};
\node(dots2) at (4.75,0){$\cdots$};
\node(n) at (6,0){$\bullet$};
\draw(0,0) -- (dots1);
\draw (dots1) -- (dots2);
\draw(2.5,0) -- (3.5,1) -- (3.5,0);
\draw(dots2) -- (6,0);
\node[below, scale = 0.8] at (1){$1$};
\node[below, scale = 0.8, xshift = 3pt] at (l-k){$\ell+1$};
\node[right, scale = 0.8] at (twin){$(\ell+1)'$};
\node[below,scale = 0.8, xshift = -2pt] at (l-k-1){$\ell$};
\node[right] at (3.5,0.5){$\epsilon_1$};
\node[above, yshift = -3pt] at (3,0){$\epsilon_2$};
\node[left] at (3.25,0.75){$\epsilon_3$};
\node[below] at (n){$n$};
\node[] at (1,0.75){$\widetilde{T}_{n,\ell}$};
\end{tikzpicture}&
\begin{tikzpicture}
\node(dots1) at (1.25,0){$\cdots$};
\node(1) at (0,0){$\bullet$};
\node(l-k-1) at (2.5,0){$\bullet$};
\node(l-k) at (3.5,0){$\bullet$};
\node(twin) at (3.5,1){$\bullet$};
\node(dots2) at (4.75,0){$\cdots$};
\node(n) at (6,0){$\bullet$};
\draw(0,0) -- (dots1);
\draw (dots1) -- (dots2);
\draw(2.5,0) -- (3.5,1) -- (3.5,0);
\draw(dots2) -- (6,0);
\node[below, scale = 0.8] at (1){$1$};
\node[below, scale = 0.8, xshift = 3pt] at (l-k){$\ell+1$};
\node[right, scale = 0.8] at (twin){$(\ell+1)'$};
\node[below,scale = 0.8, xshift = -2pt] at (l-k-1){$\ell$};
\node[right] at (3.5,0.5){$\epsilon_3$};
\node[above, yshift = -3pt] at (3,0){$\epsilon_2$};
\node[left] at (3.25,0.75){$\epsilon_1$};
\node[below] at (n){$n$};
\node[] at (1,0.75){$\widetilde{T}_{n,\ell}$};
\end{tikzpicture}\\
\begin{tikzpicture}
\node(dots1) at (1.25,0){$\cdots$};
\node(1) at (0,0){$\bullet$};
\node(l-k-1) at (2.5,0){$\bullet$};
\node(l-k) at (3.5,0){$\bullet$};
\node(twin) at (3.5,1){$\bullet$};
\node(dots2) at (4.75,0){$\cdots$};
\node(n) at (6,0){$\bullet$};
\draw(0,0) -- (dots1);
\draw (dots1) -- (dots2);
\draw(2.5,0) -- (3.5,1);
\draw(dots2) -- (6,0);
\node[below, scale = 0.8] at (1){$1$};
\node[below, scale = 0.8, xshift = 3pt] at (l-k){$\ell+1$};
\node[right, scale = 0.8] at (twin){$(\ell+1)'$};
\node[below,scale = 0.8, xshift = -2pt] at (l-k-1){$\ell$};
\node[below] at (n){$n$};
\node[] at (1,0.75){$F_{n,\ell}$};
\end{tikzpicture}&
\begin{tikzpicture}
\node(dots1) at (1.25,0){$\cdots$};
\node(1) at (0,0){$\bullet$};
\node(l-k-1) at (2.5,0){$\bullet$};
\node(l-k) at (3.5,0){$\bullet$};
\node(twin) at (3.5,1){$\bullet$};
\node(dots2) at (4.75,0){$\cdots$};
\node(n) at (6,0){$\bullet$};
\draw(0,0) -- (dots1);
\draw(dots1)--(dots2);
\draw(3.5,0) -- (3.5,1);
\draw(dots2) -- (6,0);
\node[below, scale = 0.8] at (1){$1$};
\node[below, scale = 0.8, xshift = 3pt] at (l-k){$\ell+1$};
\node[below,scale = 0.8, xshift = -2pt] at (l-k-1){$\ell$};
\node[right, scale = 0.8] at (twin){$(\ell+1)'$};
\node[below] at (n){$n$};
\node[] at (1,0.75){$F_{n,\ell+1}$};
\end{tikzpicture}\\
\begin{tikzpicture}
\node(dots1) at (1.25,0){$\cdots$};
\node(1) at (0,0){$\bullet$};
\node(l-k-1) at (2.5,0){$\bullet$};
\node(l-k) at (3.5,0){$\bullet$};
\node(twin) at (3.5,1){$\bullet$};
\node(dots2) at (4.75,0){$\cdots$};
\node(n) at (6,0){$\bullet$};
\draw(0,0) -- (dots1);
\draw (dots1) -- (2.5,0);
\draw(3.5,0) -- (dots2);
\draw(2.5,0) -- (3.5,1) --(3.5,0);
\draw(dots2) -- (6,0);
\node[below, scale = 0.8] at (1){$1$};
\node[below, scale = 0.8, xshift = 3pt] at (l-k){$\ell+1$};
\node[right, scale = 0.8] at (twin){$(\ell+1)'$};
\node[below,scale = 0.8, xshift = -2pt] at (l-k-1){$\ell$};
\node[below] at (n){$n$};
\node[] at (1,0.75){$P_{n+1}$};
\end{tikzpicture}&
\begin{tikzpicture}
\node(dots1) at (1.25,0){$\cdots$};
\node(1) at (0,0){$\bullet$};
\node(l-k-1) at (2.5,0){$\bullet$};
\node(l-k) at (3.5,0){$\bullet$};
\node(twin) at (3.5,1){$\bullet$};
\node(dots2) at (4.75,0){$\cdots$};
\node(n) at (6,0){$\bullet$};
\draw(0,0) -- (dots1);
\draw (dots1) -- (2.5,0);
\draw(3.5,0) -- (dots2);
\draw(2.5,0) -- (3.5,1) -- (3.5,0);
\draw(dots2) -- (6,0);
\node[below, scale = 0.8] at (1){$1$};
\node[below, scale = 0.8, xshift = 3pt] at (l-k){$\ell+1$};
\node[right, scale = 0.8] at (twin){$(\ell+1)'$};
\node[below,scale = 0.8, xshift = -2pt] at (l-k-1){$\ell$};
\node[below] at (n){$n$};
\node[] at (1,0.75){$P_{n+1}$};
\end{tikzpicture}\\
\begin{tikzpicture}
\node(dots1) at (1.25,0){$\cdots$};
\node(1) at (0,0){$\bullet$};
\node(l-k-1) at (2.5,0){$\bullet$};
\node(l-k) at (3.5,0){$\bullet$};
\node(twin) at (3.5,1){$\bullet$};
\node(dots2) at (4.75,0){$\cdots$};
\node(n) at (6,0){$\bullet$};
\draw(0,0) -- (dots1);
\draw (dots1) -- (2.5,0);
\draw(3.5,0) -- (dots2);
\draw(2.5,0) -- (3.5,1);
\draw(dots2) -- (6,0);
\node[below, scale = 0.8] at (1){$1$};
\node[below, scale = 0.8, xshift = 3pt] at (l-k){$\ell+1$};
\node[right, scale = 0.8] at (twin){$(\ell+1)'$};
\node[below,scale = 0.8, xshift = -2pt] at (l-k-1){$\ell$};
\node[below] at (n){$n$};
\node[] at (0.25,0.75){$P_{\ell+1} \sqcup P_{n-\ell}$};
\node[circle, radius = 1pt, fill = white] at (6.5,0){};
\end{tikzpicture}&
\begin{tikzpicture}
\node(dots1) at (1.25,0){$\cdots$};
\node(1) at (0,0){$\bullet$};
\node(l-k-1) at (2.5,0){$\bullet$};
\node(l-k) at (3.5,0){$\bullet$};
\node(twin) at (3.5,1){$\bullet$};
\node(dots2) at (4.75,0){$\cdots$};
\node(n) at (6,0){$\bullet$};
\draw(0,0) -- (dots1);
\draw (dots1) -- (2.5,0);
\draw(3.5,0) -- (dots2);
\draw(3.5,1) -- (3.5,0);
\draw(dots2) -- (6,0);
\node[below, scale = 0.8] at (1){$1$};
\node[below, scale = 0.8, xshift = 3pt] at (l-k){$\ell+1$};
\node[right, scale = 0.8] at (twin){$(\ell+1)'$};
\node[below,scale = 0.8, xshift = -2pt] at (l-k-1){$\ell$};
\node[below] at (n){$n$};
\node[] at (0.25,0.75){$P_{\ell} \sqcup P_{n-\ell+1}$};
\node[circle, radius = 1pt, fill = white] at (6.5,0){};
\end{tikzpicture}
\end{tabular}
    
\caption{Illustration of~\Cref{lemma: left-and-right-shift}.}\label{fig:LeftAndRight}
\end{figure}

\begin{lemma}[Left and Right Shift Lemma]\label{lemma: left-and-right-shift}
For $n\geq 3$ and $2\leq \ell \leq n-1$, we have 
\[
\begin{array}{cc|cc}
X_{\widetilde{T}_{n, \ell}} = X_{F_{n, \ell }} + X_{P_{n+1}} - X_{P_{\ell +1}} X_{P_{n-\ell}}& \quad & \quad & X_{\widetilde{T}_{n, \ell}} = X_{F_{n, \ell +1}} + X_{P_{n+1}} - X_{P_{\ell }} X_{P_{n - \ell + 1}}\\[0.25cm]
\text{\textit{Left shift}} & & & \text{\textit{Right shift}}
\end{array} 
\] 
\end{lemma}

Our next step is to use~\Cref{lemma: left-and-right-shift} to obtain a formula equivalent to that in~\Cref{firststep} which does not involve twinning paths. 

\begin{prop}\label{JM}
    For $n\geq 3$ and $2\leq \ell\leq n-1$, we have
    \[X_{P_{n, \ell}} = -2X_{P_{\ell - 1}} X_{P_{n - \ell +2}} + 2e_1 X_{P_n} + 4X_{P_{n+1}} - 2X_{P_{\ell}} X_{P_{n - \ell+1}} + 2e_2 X_{P_{\ell - 1}}  X_{P_{n -\ell}} -2X_{P_{\ell+1}} X_{P_{n - \ell}}.\]
\end{prop}

\begin{proof}
    Applying the Left and Right Shift Lemmas at $\ell - k- 1$ implies that 
    \begin{equation}\label{eq:left-right-together}
        X_{F_{n,\ell-k}} = X_{F_{n,\ell-k-1}} + X_{P_{\ell-k-1}}  X_{P_{n-\ell+k+2}} - X_{P_{\ell-k}} X_{P_{n-\ell+k+1}}.
    \end{equation}
    By applying~\eqref{eq:left-right-together} repeatedly, we get that 
    \begin{equation}\label{eq:PnlSumFormula}
        X_{F_{n,\ell}} = X_{F_{n,1}} + \sum_{k=0}^{\ell-2}  \left( X_{P_{\ell-k-1}}  X_{P_{n-\ell+k+2}} - X_{P_{\ell-k}} X_{P_{n-\ell+k+1}}\right).
    \end{equation}
    Since $F_{n,1}$ is precisely $P_{n+1}$, we can telescope the sum in~\eqref{eq:PnlSumFormula} to obtain
    \begin{equation}\label{telescope}
        X_{F_{n,\ell}} = X_{P_{n+1}} +   X_{P_{1}}  X_{P_{n}} - X_{P_{\ell}} X_{P_{n-\ell+1}}.
    \end{equation}

   Recall the formula in~\Cref{firststep}:
    \begin{equation*}
        X_{P_{n,\ell}} = 2X_{\widetilde{T}_{n,\ell-1}}-X_{\widetilde{T}_{\ell,\ell-1}} X_{P_{n-\ell}}.
    \end{equation*}

    We apply the Left Shift Lemma to $X_{\widetilde{T}_{n, \ell -1}}$ and the Right Shift Lemma to $X_{\widetilde{T}_{\ell, \ell - 1}}$, and obtain
    \begin{align*}
        X_{P_{n,\ell}} =  2\left(X_{F_{n, \ell-1}} + X_{P_{n+1}} - X_{P_{\ell}}  X_{P_{n - \ell + 1}}\right) 
         - \left(X_{F_{\ell, \ell}} + X_{P_{\ell + 1}} - X_{P_{\ell -1}} X_{P_2}\right) X_{P_{n - \ell}}.
    \end{align*}
    Since $F_{\ell,\ell}$ is $P_{\ell + 1}$, we can rewrite the last equation as: 
    \begin{align*}
        X_{P_{n,\ell}} &= 2\left(X_{F_{n, \ell-1}} + X_{P_{n+1}} - X_{P_{\ell}}  X_{P_{n - \ell + 1}}\right) - \left(2X_{P_{\ell + 1}} - X_{P_{\ell -1}} X_{P_2}\right) X_{P_{n - \ell}}\\
        &= 2X_{F_{n, \ell-1}} + \left[ 2X_{P_{n + 1}} - 2X_{P_{\ell}} X_{P_{n - \ell + 1}} + X_{P_{\ell - 1}} X_{P_2} X_{P_{n-\ell}} - 2X_{P_{\ell + 1}}  X_{P_{n - \ell}}\right].
    \end{align*}
    Substituting in (\ref{telescope}) for $X_{F_{n, \ell - 1}}$, we have:
    \begin{align*}
        X_{P_{n,\ell}} = & 2\left(X_{P_{n+1}} + X_{P_1} X_{P_n} - X_{P_{\ell - 1}}  X_{P_{n -\ell}}\right) \\
        &\quad+ \left[ 2X_{P_{n + 1}} - 2X_{P_{\ell}} X_{P_{n - \ell + 1}} + X_{P_{\ell - 1}} X_{P_2} X_{P_{n-\ell}} - 2X_{P_{\ell + 1}}  X_{P_{n - \ell}}\right].
    \end{align*}
    
    Finally, the formula in the statement follows by collecting all the terms and evaluating $X_{P_1} = e_1$ and $X_{P_2} = 2e_2$.
\end{proof}

Now we investigate the generating function of $X_{P_{n,\ell}}$. For this, we introduce two families of polynomials in the variable $z$ with coefficients in the ring of symmetric functions.
    
\begin{definition}\label{def: f_l and g_l}
\begin{enumerate}[(a)]
\item For $\ell\ge 2$, we define the following polynomial of degree $\ell+1$ in $z$:
$$f_\ell(z):=  2+e_1z
-X_{P_{\ell-1}}z^{\ell-1}(1-e_2z^2 ) -X_{P_\ell}z^\ell   -X_{P_{\ell+1}}z^{\ell+1} .$$ 
\item For $\ell\ge 2$, we define the following polynomial of degree $\ell+1$  :
\[
g_\ell(z):=  -\sum_{j=0}^{\ell}X_{P_j}z^j -(1+e_1z)\sum_{j=0}^{\ell-2}X_{P_j}z^j 
{-(X_{P_{\ell+1}} -e_2 X_{P_{\ell-1}}) z^{\ell+1}}.
\]

\end{enumerate}
\end{definition}
The following result gives an identity for the generating function for the chromatic symmetric function of the twinned path in terms of the generating function for the chromatic symmetric function of the path and the new families of polynomials introduced. 

\begin{prop}\label{prop:gf-twin-path-int-verticeS}  Let $2\le\ell\le n-1$. 
 The generating function for the chromatic symmetric function of the twinned path $P_{n,\ell}$, twinned at vertex $\ell$, can be written in terms of the path generating function $\cX_P$ as follows:
\begin{align}\label{eqn:KEY-twin-path-gfS}
\sum_{n\ge \ell+1} X_{P_{n,\ell} } z^{n+1} =2\cX_P(z) f_\ell(z) +2g_\ell(z). 
\end{align}
\end{prop}

\begin{proof} 
We reorder the terms appearing in the recurrence in~\Cref{JM} to make the source of the factor $f_\ell(z)$ accompanying $\cX_P$ more transparent: 
\begin{equation}\label{eqn:Reorder-twin-path-deg2-recS}
X_{P_{n,\ell} }= 4X_{P_{n+1}}+ 2e_1 X_{P_n} 
-2X_{P_{\ell - 1}} (X_{P_{n - \ell +2}} -e_2  X_{P_{n -\ell}})  - 2X_{P_{\ell}} X_{P_{n - \ell+1}}  -2X_{P_{\ell+1}} X_{P_{n - \ell}}.
\end{equation}
Multiplying by $z^{n+1}$ and summing over $n\ge \ell+1$ gives 
\begin{align*}
&\sum_{n\ge \ell+1} X_{P_{n,\ell}} z^{n+1}\\
 &=
2\big[2+e_1z
-X_{P_{\ell-1}}z^{\ell-1}(1-e_2z^2 ) -X_{P_\ell}z^\ell   -z^{\ell+1} X_{P_\ell+1} \big]\cX_P\notag\\
&   -4 \sum_{j=0}^\ell X_{P_j}z^j  -2z e_1 \sum_{j=0}^{\ell-1} X_{P_j}z^j  +2z^{\ell-1} (1+z e_1) X_{P_{\ell-1}} +2 X_{P_\ell}z^\ell  {-2(X_{P_{\ell+1}} -e_2 X_{P_{\ell-1}}) z^{\ell+1}}.
\end{align*}
\noindent
Here we have made the substitutions  
$\sum_{j=0}^{2} X_{P_j} z^j=1 + e_1 z+2 e_2 z^2$, $\sum_{j=0}^1 X_{P_j} z^j=1+e_1z$ and $X_{P_0}=1$.

The expression for $f_\ell(z)$ follows immediately from the first line above.

Now rewrite the second line as 
\begin{multline*}
 -4X_{P_\ell}z^\ell -4 X_{P_{\ell-1}}z^{\ell-1}    -4 \sum_{j=0}^{\ell-2} X_{P_j}z^j  -2z^{\ell} e_1 X_{P_{\ell-1}}-2z e_1 \sum_{j=0}^{\ell-2} X_{P_j}z^j \\ + 2X_{P_{\ell-1}}z^{\ell-1}  +2 e_1 X_{P_{\ell-1}}z^{\ell}  +2 X_{P_\ell}z^\ell 
 {-2(X_{P_{\ell+1}} -e_2 X_{P_{\ell-1}}) z^{\ell+1}}
 ,\end{multline*}
which in turn yields the  expression for $g_\ell(z)$ in \Cref{def: f_l and g_l}. 
\end{proof}

Although 
$f_\ell(z)$ is not $e$-positive, 
we can conclude the following.
\begin{cor}\label{cor:key-fact-fell-polyS}
    The $e$-positivity of $X_{P_{n,\ell}}$ is equivalent to the $e$-positivity of $\cX_P(z) f_\ell(z)$.
\end{cor}
\begin{proof} The degree of $g_\ell(z)$ as a polynomial in $z$ is $\ell+1$, while the left-hand side of~\eqref{eqn:KEY-twin-path-gfS} has lowest degree $\ell+2$ in $z$. We conclude that all terms in $g_\ell(z)$ are necessarily canceled out by identical terms in $\cX_P(z) f_\ell(z)$.
\end{proof}

Our next result rewrites $f_\ell(z)$ as a positive expansion of other functions that were introduced in~\eqref{eqn:KEY-epos-serieS}.
\begin{lemma}\label{lemma: fl formula}
For $\ell\geq 2$, we have
     \begin{equation}\label{eq: claim fell}
     f_\ell(z)=\sum_{i=3}^{\ell+1}(i-2)e_iz^i+2(D+G_{\geq \ell+2})+\sum_{i=1}^{\ell-2}(D+G_{\geq \ell+2-i})X_{P_i}z^i.
     \end{equation}
\end{lemma}
\begin{proof}
By~\Cref{def: f_l and g_l},
\[f_\ell(z)=  2+e_1z
+X_{P_{\ell-1}}e_2z^{\ell+1}
-X_{P_{\ell-1}}z^{\ell-1} -X_{P_\ell}z^\ell   -X_{P_{\ell+1}}z^{\ell+1} .
\]

For $\ell=2$, recall that $X_{P_3}=e_2e_1+3e_3$, and  $X_{P_1}=e_1$ and $X_{P_2}=2e_2$. Then we have
\begin{equation*}
\begin{split}
f_2(z)&=2-2e_2z^2-3e_3z^3=2(1-e_2z^2-2e_3z^3) +e_3z^3 \\
&= 2(1-G(z)+G_{\geq 4}(z))+e_3z^3=2(D(z)+G_{\ge 4}(z))+e_3z^3 .
\end{split}
\end{equation*}

Let $\phi_\ell$ denote the right-hand side of~\eqref{eq: claim fell}. We show that $\phi_\ell$ and $f_\ell$ satisfy the same recurrence relation.
It is straightforward to see that for $\ell \ge 2$,
\begin{equation}\label{eq: fell claim lhs-1}
f_{\ell+1}-f_\ell=
X_{P_{\ell}}e_2z^{\ell+2}-X_{P_{\ell+2}}z^{\ell+2}-X_{P_{\ell-1}}e_2z^{\ell+1}+X_{P_{\ell-1}}z^{\ell-1}.
\end{equation}
Next we look at $\phi_{\ell+1}-\phi_\ell$. 
Observe from~\eqref{eqn:KEY-epos-serieS} that   $G_{\geq m+1}-G_{\geq m}=-(m-1)e_mz^m$. We therefore have 
\begin{align*}
    \phi_{\ell+1}&=\sum_{i=3}^{\ell+2} (i-2)e_i z^i +2 (D+G_{\geq \ell+3} ) 
    +\sum_{i=1}^{\ell-1} (D+G_{\ell+3-i} )X_{P_i}z^i\\
    \phi_{\ell}&=\sum_{i=3}^{\ell+1} (i-2)e_i z^i +2 (D+G_{\geq \ell+2} ) 
    +\sum_{i=1}^{\ell-2} (D+G_{\ell+2-i} ) X_{P_i}z^i
\end{align*}
and hence we obtain, for $\ell\ge 2$, 
\[\phi_{\ell+1}-\phi_{\ell}=\ell e_{\ell+2}z^{\ell+2} -2 (\ell+1) e_{\ell+2}z^{\ell+2} -\sum_{i=1}^{\ell-2} (\ell+1-i) e_{\ell+2-i} z^{\ell+2-i} X_{P_i}z^i +(D+G_{\ge 4}) X_{P_{\ell-1}} z^{\ell-1}.
\]
The path recurrence relation in~\Cref{prop:RECURSION-pathS-cycleS}\ref{RECURSION-pathS} tells us that 
\[X_{P_{\ell+2}}=(\ell+2)e_{\ell+2} +\sum_{j=1}^{\ell-2} (\ell+1-j) e_{\ell+2-j} X_{P_j} +2e_3 X_{P_{\ell-1}} + e_2 X_{P_\ell}.\]
Together with  $D+G_{\ge 4}=1-G_{\le 3}=1-e_2z^2-2e_3z^3$, this gives
\[\phi_{\ell+1}-\phi_{\ell}= (-X_{P_{\ell+2}} +2e_3 X_{P_{\ell-1}} + e_2 X_{P_\ell}) z^{\ell+2} + (1-e_2 z^2-2e_3 z^3) X_{P_{\ell-1}} z^{\ell-1}.
\]
The terms containing $e_3 X_{P_{\ell-1}}$ cancel, and the remaining expression coincides with the one for $f_{\ell+1}-f_\ell$ in~\eqref{eq: fell claim lhs-1}.
Hence $\phi_\ell$ and $f_\ell$ satisfy the same recurrence relation. Since their initial values also coincide, the claim follows.\qedhere
\end{proof}

The preceding efforts culminate in the following $e$-positivity result, as announced at the start of this section.
\begin{theorem}\label{thm:cX_P-f_ell-formula} For $\ell\geq 2$, we have the $e$-positive expansion 
\begin{equation}\label{eq:cX_p-f_ell-formula}
\cX_P f_\ell=\cX_P\sum_{i=3}^{\ell+1}(i-2)e_iz^i+2(E+\cX_P G_{\geq \ell+2})+\sum_{i=1}^{\ell-2}(E+\cX_P G_{\geq \ell+2-i})X_{P_i}z^i.\end{equation}
Hence the generating function $\sum_{n\geq \ell+1}X_{P_{n,\ell}}z^{n+1}$ is $e$-positive.
\end{theorem}
\begin{proof}
The expression for $f_\ell$ in~\Cref{lemma: fl formula} immediately allows us to conclude~\eqref{eq:cX_p-f_ell-formula}, using the fact that $\cX_P(z) D(z) =E(z)$.
    It then suffices to observe that the generating function $\sum_{n\geq \ell+1}X_{P_{n,\ell}}z^{n+1}$ is  comprised precisely of all the terms of degree $\geq \ell+2$ in the $e$-positive rational expression \begin{align*}
    2\cX_Pf_\ell&=\frac{2Ef_\ell}{D}\\
    &=\frac{\displaystyle{2E\left(\sum_{i=3}^{\ell+1}(i-2)e_iz^i+G_{\geq \ell+2}+\sum_{i=0}^{\ell-2}G_{\geq \ell+2-i}X_{P_i}z^i\right)}}{1-\sum_{i\geq 2}(i-1)e_iz^i}+2(1+E)\sum_{i=0}^{\ell-2}X_{P_i}z^i. \qedhere
\end{align*}
\end{proof}

From \Cref{thm:cX_P-f_ell-formula} and a tedious computation of $\cX_Pf_\ell+g_\ell$, we obtain the following cancellation-free $e$-positive expression for 
$\frac{1}{2}\sum_{n\geq \ell+1}X_{P_{n,\ell}}z^{n+1}$.  (Note that the sum is zero if the range of summation is empty.)

\begin{prop}\label{theorem:twin-path-epoS-vertex} 
For integers $n\geq 3$ and  $2\leq \ell\leq  n-1$, the twin $P_{n,\ell}$ of the path $P_n$ at the degree 2 vertex $\ell$  
is $e$-positive. In particular, we have 
\begin{align*}
  \frac{1}{2}  \sum_{n\geq \ell+1} X_{P_{n,\ell}} z^{n+1}
    &=\ell e_{\ell+1}z^{\ell+1}\lp\sum_{i=1}^{\ell-2}X_{P_i}z^i\rp+\sum_{i=3}^{\ell}(i-1)e_iz^i\lp\sum_{j=0}^{i-4}X_{P_{\ell-2-j}}z^{\ell-2-j}\rp+E_{\geq \ell+2}\\
    &\quad+E_{\geq\ell+ 2}\sum_{i=0}^{\ell-2}X_{P_i}z^i 
    +\lp \sum_{i\geq \ell-1}X_{P_i}z^i\rp \lp \sum_{i=2}^{\ell+1}(i-2)e_iz^i\rp  \\
    &\quad+2\cX_PG_{\geq \ell+2}+\cX_P\sum_{i=1}^{\ell-2}G_{\geq \ell+2-i} X_{P_i}z^i.
\end{align*}
\end{prop}

\subsection{Generating function for twinned cycles}
\label{sec:gf-twin-cycle}

In this section, we  establish a new  result, the $e$-positivity of the chromatic symmetric function of the twinned cycle.
Again, our goal of obtaining an $e$-positive generating function for the twinned cycle will begin with a formula for the chromatic symmetric function of the twinned cycle, which is derived using the triple deletion formula.

We start by introducing two more families of graphs. 
Consider the twinned cycle graph $C_{n,v}$ where $v'$ is the twinned vertex of $v$ and $u$ and $w$ are the adjacent vertices to $v$ and $v'$.
Let $D_{n+1}$ be the graph obtained from $C_{n,v}$ by removing the edge $uv$ and let $\Tad_{n+1}$    
be the graph obtained from $C_{n,v}$ by removing the edges $uv$ and $vv'$.  
We illustrate these two definitions in~\Cref{fig: 3 graphs for twinned cycle}.

\begin{figure}[ht]
    \centering
        \begin{subfigure}[b]{0.3\textwidth}
    \centering
\begin{tikzpicture}[scale = 0.7]
\draw(0,2.7) -- (0,1.7);
\draw(0,2.7) -- (1.4,1.4);
\draw(0,1.7) -- (1.4,1.4);
\draw(1.4,1.4) -- (2,0);
\draw(2,0) -- (1.4,-1.4);
\draw(1.4,-1.4) -- (0.5,-2);
\draw(-0.6,-2) --(-1.4,-1.4);
\draw(-1.4,-1.4) -- (-2,0);
\draw(-2,0) -- (-1.4,1.4);
\draw(-1.4,1.4) -- (0,2.7);
\draw(-1.4,1.4) -- (0,1.7);
\node(1) at (0,2.7){$\bullet$};
\node(1twin) at (0,1.7){$\bullet$};
\node(2) at (1.4,1.4){$\bullet$};
\node(3) at (2,0){$\bullet$};
\node(4) at (1.4,-1.4){$\bullet$};
\node(5) at (0,-2){$\cdots$};
\node(8) at (-1.4,1.4){$\bullet$};
\node(7) at (-2,0){$\bullet$};
\node(6) at (-1.4,-1.4){$\bullet$};
\node[above, yshift = 2pt] at (1){$v$};
\node[below] at (1twin){$v'$};
\node[left, xshift = -2pt] at (8){$u$};
\node[right, xshift = 2pt] at (2){$w$};
\end{tikzpicture}
\caption{$C_{n,v}$}
\end{subfigure}
    \begin{subfigure}[b]{0.3\textwidth}
    \centering
\begin{tikzpicture}[scale = 0.7]
\draw(0,2.7) -- (0,1.7);
\draw(0,2.7) -- (1.4,1.4);
\draw(0,1.7) -- (1.4,1.4);
\draw(1.4,1.4) -- (2,0);
\draw(2,0) -- (1.4,-1.4);
\draw(1.4,-1.4) -- (0.5,-2);
\draw(-0.6,-2) --(-1.4,-1.4);
\draw(-1.4,-1.4) -- (-2,0);
\draw(-2,0) -- (-1.4,1.4);
\draw(-1.4,1.4) -- (0,1.7);
\node(1) at (0,2.7){$\bullet$};
\node(1twin) at (0,1.7){$\bullet$};
\node(2) at (1.4,1.4){$\bullet$};
\node(3) at (2,0){$\bullet$};
\node(4) at (1.4,-1.4){$\bullet$};
\node(5) at (0,-2){$\cdots$};
\node(8) at (-1.4,1.4){$\bullet$};
\node(7) at (-2,0){$\bullet$};
\node(6) at (-1.4,-1.4){$\bullet$};
\node[above, yshift = 2pt] at (1){$v$};
\node[below] at (1twin){$v'$};
\node[left, xshift = -2pt] at (8){$u$};
\node[right, xshift = 2pt] at (2){$w$};
\end{tikzpicture}
\caption{$D_{n+1}$}
\end{subfigure}
\begin{subfigure}[b]{.3\textwidth}
\centering
\begin{tikzpicture}[scale = 0.7]
\draw(0,2.7) -- (1.4,1.4);
\draw(0,1.7) -- (1.4,1.4);
\draw(1.4,1.4) -- (2,0);
\draw(2,0) -- (1.4,-1.4);
\draw(1.4,-1.4) -- (0.5,-2);
\draw(-0.6,-2) --(-1.4,-1.4);
\draw(-1.4,-1.4) -- (-2,0);
\draw(-2,0) -- (-1.4,1.4);
\draw(-1.4,1.4) -- (0,1.7);
\node(1) at (0,2.7){$\bullet$};
\node(1twin) at (0,1.7){$\bullet$};
\node(2) at (1.4,1.4){$\bullet$};
\node(3) at (2,0){$\bullet$};
\node(4) at (1.4,-1.4){$\bullet$};
\node(5) at (0,-2){$\cdots$};
\node(8) at (-1.4,1.4){$\bullet$};
\node(7) at (-2,0){$\bullet$};
\node(6) at (-1.4,-1.4){$\bullet$};
\node[above, yshift = 2pt] at (1){$v$};
\node[below] at (1twin){$v'$};
\node[left, xshift = -2pt] at (8){$u$};
\node[right, xshift = 2pt] at (2){$w$};
\end{tikzpicture}
\caption{$\Tad_{n+1}$}
\end{subfigure}
    \caption{}
    \label{fig: 3 graphs for twinned cycle}
\end{figure}
\begin{lemma}\label{lemma:twin-cycle-TD}  
For $n\geq 3$: 
\begin{equation*}
X_{C_{n,v}}
   =4 X_{C_{n+1}}+2 e_1 X_{C_n} -6X_{P_{n+1}} +2e_2 X_{P_{n-1}}.
   \end{equation*}
\end{lemma}
\begin{proof} 
Consider $n\geq 3$. By the triple deletion argument applied to $\epsilon_1=uv$ and $\epsilon_2=uv'$, we get that
\begin{equation}\label{eqn:2ndL-triple-deletion-cycle}
X_{C_{n,v}}=2 X_{D_{n+1}} -X_{P_{n,v}}=2 X_{D_{n+1}} -2X_{P_{n+1}}+X_{P_2}X_{P_{n-1}}.
\end{equation}
In $D_{n+1}$, applying the triple deletion argument to $\epsilon_{1}=vw$ and $\epsilon_2=vv'$ gives
\begin{equation}\label{eqn:1st-triple-deletion-D1}
X_{D_{n+1}} =2X_{\Tad_{n+1}}-e_1 X_{C_n},
\end{equation}
while applying the triple deletion argument to $\epsilon_{1}=vw$ and $\epsilon_2=v'w$ gives
\begin{equation*}
X_{D_{n+1}} =X_{\Tad_{n+1}}+X_{C_{n+1}}-X_{P_{n+1}}.
\end{equation*}
Subtracting both expressions for $X_{D_{n+1}}$ we obtain that
\begin{equation*}
X_{\Tad_{n+1}}=X_{C_{n+1}}+e_1 X_{C_n} -X_{P_{n+1}},
\end{equation*}
and therefore
\begin{equation}\label{eqn:D1}
X_{D_{n+1}}=2X_{C_{n+1}}+e_1 X_{C_n} -2X_{P_{n+1}}.
\end{equation}

Finally, putting together~\eqref{eqn:2ndL-triple-deletion-cycle} and~\eqref{eqn:D1}, we have
\begin{equation*}
   X_{C_{n,v}}
   =4 X_{C_{n+1}}+2 e_1 X_{C_n} -6X_{P_{n+1}} +2e_2 X_{P_{n-1}},
   \end{equation*}
   as claimed.  
   \end{proof}

Let $\cX_{C_v}$ be the generating function for the twinned cycle, that is, $\cX_{C_v}(z) := \sum_{n\ge 3} X_{C_{n,v}} z^{n+1}$. By~\Cref{lemma:twin-cycle-TD} we have the following expression for $\cX_{C_v}$.
   \begin{cor}\label{cor:twin-cycle-TD-GF}
    The generating function of the twinned cycle can be written as 
\begin{equation*}
\cX_{C_v}(z) 
=2(2+e_1z)\cX_C-2(3-e_2z^2)\cX_P +6(1+e_1z) +2e_2z^2 -6e_3z^3 .
\end{equation*}   
   \end{cor}
   \begin{proof}
   The generating function  follows by multiplying the formula in~\Cref{lemma:twin-cycle-TD} by $z^{n+1}$ and summing over all $n\ge 3$. In particular, taking into account the initial terms that do not appear and using the initial values 
   $X_{C_1}=0, X_{P_1}=e_1, X_{C_2}=2e_2=X_{P_2},$ $X_{C_3}=6e_3,$ and $ X_{P_3}=e_2e_1+3e_3$, we get the following expressions in terms of  the generating functions for the cycle and the path:
   \begin{enumerate}[(a)]
   \itemsep0.3cm
   \item
    $4\sum_{n\ge 3} X_{C_{n+1}} z^{n+1}=4(\cX_C-z^2X_{C_2}-z^3 X_{C_3})$, 
    \item
   $2(e_1z)\sum_{n\ge 3} X_{C_n} z^n=2e_1z(\cX_C-z^2 X_{C_2})$,
   \item
   $6\sum_{n\ge 3} X_{P_{n+1}} z^{n+1}=6(\cX_P-1-z X_{P_1}-z^2 X_{P_2} -z^3 X_{P_3})$, and  
   \item
   $2e_2z^2 \sum_{n\ge 3} X_{P_{n-1}}z^{n-1}=2e_2z^2 (\cX_P-1-z X_{P_1})$.  
   \end{enumerate}
   Putting all this together gives the generating function as stated. \qedhere
   \end{proof}

Consider the following  $e$-positive generating functions that appear in the proof of~\Cref{lem:Ez-identitieS}:
\[F_2=\sum_{i\ge 3} (2i^2-5i) e_i z^i\qquad \text{and} \qquad F_3 = \sum_{i\ge 4} [(i-1)(i-3)] e_i z^i.\]
Our goal is to show that the expression given in~\Cref{{cor:twin-cycle-TD-GF}} is indeed $e$-positive.

\begin{lemma}\label{lem:Rewrite-Cnv-to-find-coefficients}
The twinned cycle generating function, scaled by $\frac12$, can be written as \[
\frac12 \cX_{C_{v}} = \frac{1}{D(z)} [F_2+e_1zF_3 +e_2z^2 (E-1-e_1z)  ]-\frac{e_2z^2}{D(z)} +e_2z^2  -3 e_3z^3 .
\]
\end{lemma}

\begin{proof}
By~\Cref{cor:twin-cycle-TD-GF},
\begin{equation}\label{eqn:gf-twin-cycleS2}
\frac{1}{2}\cX_{C_{v}}=(2+e_1z)\cX_C -3(\cX_P-1-e_1z)+e_2z^2(\cX_P+1) -3e_3z^3.
\end{equation}
From the proof of~\Cref{lem:Ez-identitieS}, we have the following \[ \cX_P-(1+e_1z)=\dfrac{z(1+e_1z)E'(z)-e_1zE(z)}{D(z)}\quad\text{and}\quad \cX_C=\frac{z^2E''(z)}{D(z)},\]
and by definition $\cX_P(z) = \frac{E(z)}{D(z)}$ and $\cX_C(z) = \frac{z^2 E''(z)}{D(z)}$.  Substituting these into~\eqref{eqn:gf-twin-cycleS2}, we get
\begin{align*}
 \frac{1}{2} \cX_{C_{v}} &=\frac{1}{D(z)} [(2+e_1z)z^2E''(z) -3(zE'(z)+z^2 e_1 E'(z) -e_1zE(z) ) +e_2z^2 E(z)] +e_2z^2  -3e_3z^3 \\
&= \frac{1}{D(z)} [(2 z^2 E''-3zE') +e_1z (z^2 E''-3e_1z E' +3E) +e_2z^2  E(z) ]+e_2z^2  -3e_3z^3 \\
&=\frac{1}{D(z)} [F_2-3e_1z-2e_2z^2+e_1z(F_3-e_2z^2+3) +e_2z^2  E(z) ]+e_2z^2  -3e_3z^3,
\end{align*}
where the final equality comes from~\Cref{lem:Ez-identitieS}. The statement then follows after further algebraic manipulations. 
\end{proof}

\begin{theorem}\label{thm:twin-cycle-genfunction-rational-expression}
    The  generating function for $\frac{1}{2}X_{C_{n,v}}$ has the following $e$-positive rational expression:
    \[\frac{1}{2}\sum_{n\geq 3}X_{C_{n,v}}z^{n+1}=\sum_{i\geq 4}(2i^2-5i)e_iz^i+\frac{e_1 z F_3 + F_2\sum_{i\ge 3} (i-1) e_i z^i
+ e_2z^2 \sum_{i\ge 3} (2i^2-6 i+2)e_i z^i}{1-\sum_{i\geq 2}(i-1)e_iz^i}.\]
\end{theorem}

\begin{proof}
    Let $E_{\geq 2}:=E-1-e_1z=\sum_{i\geq 2}e_iz^i$. Then by~\Cref{lem:Rewrite-Cnv-to-find-coefficients}, we have 
    \begin{align}
        \frac{1}{2}\cX_{C_{v}}&=\frac{1}{D(z)}\lp F_2+e_1zF_3+e_2z^2E_{\geq 2}\rp-\frac{e_2z^2}{D(z)}+e_2z^2-3e_3z^3\nonumber\\
        &=F_2-3e_3z^3+\frac{e_1zF_3}{D(z)}+\Bigg[F_2\left(\frac{1}{D(z)}-1\right)+\frac{e_2z^2E_{\geq 2}}{D(z)}-\frac{e_2z^2}{D(z)}+e_2z^2\Bigg].\label{eq:thm-ratl-gen-funct-1}
    \end{align}
    The first two terms in~\eqref{eq:thm-ratl-gen-funct-1} can be written as:
    \begin{equation}\label{eq:thm-ratl-gen-funct-2}
        F_2-3e_3z^3=\sum_{i\geq 4}(2i^2-5i)e_iz^i.
    \end{equation} For the generating function in the brackets of~\eqref{eq:thm-ratl-gen-funct-1}, we have
    \begin{align}
        &F_2\left(\frac{1}{D(z)}-1\right)+\frac{e_2z^2E_{\geq 2}}{D(z)}-\frac{e_2z^2}{D(z)}+e_2z^2\nonumber\\
        &=F_2\sum_{k\geq 1}G^k+e_2z^2E_{\geq 2}\sum_{k\geq 0}G^k-e_2z^2\sum_{k\geq 1}G^k \nonumber \\
        &=F_2\sum_{k\geq 1}G^k+e_2z^2E_{\geq 2}\sum_{k\geq 1}G^{k-1}-e_2z^2\sum_{k\geq 1}G^k \nonumber \\
        &=\sum_{k\geq 1}G^{k-1}\lp GF_2-e_2z^2(G-E_{\geq 2})\rp  \nonumber\\
        &=\sum_{k\geq 1}G^{k-1}\lp GF_2-e_2z^2F_2+e_2z^2F_2-e_2z^2\sum_{i\geq 3}(i-2)e_iz^i\rp\nonumber \\ 
        &=\sum_{k\geq 1}G^{k-1}\lp F_2 G_{\geq 3}+e_2z^2 \sum_{i\geq 3}(2i^2-6i+2)e_iz^i\rp  \nonumber\\
        &=\frac{ F_2 G_{\geq 3}+e_2z^2 \sum_{i\geq 3}(2i^2-6i+2)e_iz^i}{D(z)},\label{eq:thm-ratl-gen-funct-3}
    \end{align}
    where the penultimate equality follows from the definitions of  $G_{\geq 3}$ and $F_2$. Combining~\eqref{eq:thm-ratl-gen-funct-1},~\eqref{eq:thm-ratl-gen-funct-2}, and~\eqref{eq:thm-ratl-gen-funct-3}, we obtain the desired expression.
\end{proof}
From the generating function in~\Cref{thm:twin-cycle-genfunction-rational-expression}, we can readily extract the $e$-coefficients of $X_{C_{n,v}}.$ 
\begin{cor}\label{cor:e-coefficients-twinned-cycle}
Let $\lambda$ be a partition of $k\ge 3$, $\lambda=\left\langle 1^{m_1},2^{m_2},\dots,k^{m_{k}}\right\rangle$,  
     and let $c_\lambda$ be the coefficient of $e_{\lambda}z^{|\lambda|}$ in $\frac{1}{2}\cX_{C_{v}}$. 
    We have the following list of expressions for the coefficients:
    \begin{enumerate}[(a)]
    \item    $c_{(k)}=k(2k-5)$.
    \item If $m_1>1$, then $c_\lambda=0.$
    \item If $m_1=1$ and $\lambda=\mu\cup 1$ (so that $m_1(\mu)=0)$, then
    \[c_\lambda=\sum_{\substack{i\ge 4\\i\in\supp(\mu)}} (i-1)(i-3)\eps(\mu-i).
\]
Note that this is 0 unless $\lambda$ has a part of size at least $4$.
\item If $m_1=m_2=0$, then 
\[c_\lambda=\sum_{a\in\supp(\lambda)} (2a^2-5a) \eps(\lambda-a).\]
\item If $m_1=0$ and $m_2=\ell(\lambda)$, then $c_\lambda=0$.
\item If $m_1=0$ and $1\leq m_2<\ell(\lambda)$, then  
\[c_\lambda=\sum_{\substack{a,b\geq 3\\(a,b)\in \supp(\lambda)}} \eps(\lambda-a-b)(2a^2-5a)(b-1)+\sum_{\substack{c\geq 3\\c\in \supp(\lambda)}}\eps(\lambda-c-2)(2c^2-6c+2)\]
where $(a,b)\in \supp(\lambda)$ means both $a$ and $b$ are in $\supp(\lambda)$ if $a\neq b$, and $m_a\geq 2$ if $a=b$.
\end{enumerate}
\end{cor}

\section{\texorpdfstring{$e$}{e}-positivity via recurrences}\label{sec:recurrences}

In this section, we reprove several $e$-positivity results for certain classes of graphs by exhibiting an $e$-positive recurrence relation. The recurrence relations for paths and cycles from~\Cref{{prop:RECURSION-pathS-cycleS}} serve as the model for those of this section, and in fact will play a key role in our derivations. 
We will also need explicit expressions for some coefficients, which are readily 
extracted from~\Cref{{prop:RECURSION-pathS-cycleS}}. These are recorded  in the next result.
\begin{cor}\label{cor:coeffS-RECURSION-pathS-cycleS} 
Given a graph $G$, let $[e_\lambda]X_{G}$ denote the coefficient of $e_\lambda$ in the chromatic symmetric function of $G$, $X_G$. 
\begin{itemize}
\itemsep0.25cm
    \item For $n\ge 2$, $[e_n]X_{P_n}=n$, $[e_{n-1}e_1]X_{P_n}=n-2$, and $[e_n]X_{C_n}=n(n-1)$.
    \item For $n\ge 5$, $[e_{n-2}e_2]X_{P_n}=3n-8$ and $[e_{n-2}e_2]X_{C_n}=n(n-3)$.
    \item For $k\ge 2$ and $r\geq 1$, $[(e_{k})^r]X_{P_{kr}}=k(k-1)^{r-1}$ and $[e_{(k^r)}]X_{C_{kr}}=k(k-1)^{r}$.
    \item $[e_2^2]X_{C_4}=2$.
\end{itemize}
\end{cor}

The rest of this section follows the structure of ~\Cref{sec: generating functions}.

\subsection{Recurrences for twinned paths}

In this section we derive  recurrence formulas for the chromatic symmetric function for a path twinned at one or both leaves or at an internal vertex.
\subsubsection{Paths twinned at a leaf} 

\hspace{5cm}

In the next proposition, we give formulas for the chromatic symmetric function $X_{P_{n,v}}$  of the path twinned at a leaf, one in terms of the path chromatic symmetric function, and the other a recurrence in the spirit of~\Cref{prop:RECURSION-pathS-cycleS}. 
 The recurrence given below makes the $e$-positivity transparent.
\begin{prop}\label{SHORTprf-epoS-twin-path-leaf} \
Let $v$ be a leaf of the path $P_n$. Then, for $n\ge 4$, 
\begin{align*}\label{eqn:twin-leaf-path-to-pathS}
X_{P_{n,v}}
&=2 (n+1) e_{n+1}
+2\sum_{j=3}^{n}(j-1) e_j X_{P_{n+1-j}}.
\end{align*}
Thus $X_{P_{n,v}}$ is $e$-positive. 
Moreover, for $n\ge 4$, $X_{P_{n,v}}$
 satisfies the $e$-positive recurrence
\[X_{P_{n,v}}=\sum_{j=2}^{n-2} (j-1) e_j  X_{P_{n-j,v}}  + 2(n+1)e_{n+1}+2(n-1)e_n e_1 + 2 (n-3) e_{n-1}e_2 , \]
with initial values $X_{P_{1,v}}= 2 e_2, \, X_{P_{2,v}}=6 e_3, \text{ and }
X_{P_{3,v}}=8 e_4+ 4  e_3e_1$. 
\end{prop}

\begin{proof}
The first expression follows from~\Cref{prop:RECURSION-pathS-cycleS}\ref{RECURSION-pathS} and~\Cref{prop:gf-twin-path-end-vertex},  noting that 
\begin{align*}
X_{P_{n,v}} &= 2X_{P_{n+1}} - X_{P_2} X_{P_{n-1}} 
= 2(n+1)e_{n+1} + 2\sum_{j=2}^n (j-1)e_jX_{P_{n+1-j}}  - 2e_2X_{P_{n-1}} \\
&=2(n+1)e_{n+1} + 2\sum_{j=3}^n (j-1)e_jX_{P_{n+1-j}} + 2e_2X_{P_{n-1}} - 2e_2X_{P_{n-1}}. 
\end{align*}

For the second recurrence,  we 
apply the triple deletion argument to $P_{n,v}$ followed by~\Cref{prop:RECURSION-pathS-cycleS}\ref{RECURSION-pathS} to both terms. Thus,
    \begin{align*}
        X_{P_{n,v}} &= 2X_{P_{n+1}} - X_{P_2} X_{P_{n-1}} = 2X_{P_{n+1}} - 2e_2 X_{P_{n-1}} \\
        &= 2(n+1)e_{n+1} + 2\sum_{j=2}^n (j-1)e_jX_{P_{n+1-j}} - 2e_2\left[(n-1)e_{n-1} + \sum_{j=2}^{n-2} (j-1)e_jX_{P_{n-1-j}}\right]\\
        &= 2(n+1)e_{n+1}-2(n-1)e_{n-1}e_2  + 2\sum_{j=2}^{n-2} (j-1)e_j \left[2X_{P_{n+1-j}} - 2e_2X_{P_{n-1-j}}\right]\\
        &\qquad+ 2(n-2)e_{n-1}X_{P_{2}} + 2(n-1)e_{n}X_{P_1}\\
        &= 2(n+1)e_{n+1} + 2\sum_{j=2}^{n-2} (j-1)e_j X_{P_{n-j,v}} + 2(n-1)e_ne_1 + 2(n-3)e_{n-1}e_{2},
    \end{align*}
    where the final step follows by the triple deletion argument applied to $P_{n-j,v}$. As this is an $e$-positive recursion with $e$-positive initial conditions, by induction it follows that $X_{P_{n,v}}$ is $e$-positive for all $n$.
\end{proof}
    \subsubsection{Paths twinned at both leaves}
Our new contribution is the $e$-positive recurrence below.

\begin{prop}\label{SHORTprf-epoS-twin-path-both-leaves} 
For $n \geq 6$, the chromatic symmetric function $X_{P_{n,v,w}}$ for the  path $P_n$ twinned at both leaves $v,w$ satisfies the recurrence 
\begin{equation*}
\begin{split}
\frac{1}{4}X_{P_{n,v,w}}&=\frac{1}{4}\sum_{j=3}^{n-3}(j-1)e_j  X_{P_{n-j,v,w}} \\
&\quad  + (n+2) e_{n+2} + n\, e_{n+1} e_1 +  3 (n-2) e_{n-1} e_3
  + 2(n-3) e_{n-2} e_3 e_1 +4(n-3) e_{n-2} e_4\\
&\quad+ e_2\left[ \frac{1}{4} X_{P_{n-2,v,w}} -2 e_n -(n-4) e_{n-2}e_2  -(n-2) e_{n-1}e_1 \right],
\end{split}
\end{equation*}
with the initial conditions 
\[
\begin{array}{lcl}
X_{P_{2,v,w}}=24 e_4, & \qquad & X_{P_{4,v,w}}=24 e_3^2+8 e_4 e_2+16 e_5 e_1 + 24 e_6, \\
X_{P_{3,v,w}}=4 e_3 e_2+12 e_4 e_1+ 20 e_5, & & X_{P_{5,v,w}}= 16e_{3}e_{3}e_{1} + 68e_{4}e_{3} + 12e_{5}e_{2}+20e_{6}e_{1}+28e_7.
\end{array}
\]
Moreover, despite the negative terms, the expression is $e$-positive.
\end{prop}

\begin{proof}
By the triple deletion argument, we have that 
\[ X_{P_{n,v,w}} = 2X_{P_{n+1,v}} - 2e_2X_{P_{n-1,v}}.\] 
Applying the triple deletion argument again to both twinned terms, we have that $X_{P_{n+1,v}} = 2X_{P_{n+2}} - 2e_2X_{P_n}$ and $X_{P_{n-1,v}} = 2X_{P_{n}} - 2e_2X_{P_{n-2}}$. Thus, for $n\geq 3$,
\begin{equation}\label{eq:useful_path_twinned_both_leaves}
    \frac{1}{4}X_{P_{n,v,w}} = X_{P_{n+2}}+ e_2^2 X_{P_{n-2}} -2e_2X_{P_n} .
\end{equation}

Now we prove the recurrence relation by strong induction on $n$. The initial conditions are checked directly. For $n\geq 6$, using repeatedly~\eqref{eq:useful_path_twinned_both_leaves} and~\Cref{prop:RECURSION-pathS-cycleS}\ref{RECURSION-pathS}, we write
\begin{align*}
    \frac{1}{4}X_{P_{n,v,w}}&=X_{P_{n+2}} -2e_2X_{P_n} + e_2^2 X_{P_{n-2}}\\
    &= (n+2)e_{n+2} + \sum_{j=2}^{n+1}(j-1)e_jX_{P_{n+2-j}} +(n-2)e_{n-2}e_2^2 + e_2^2\sum_{j=2}^{n-3}(j-1)e_jX_{P_{n-2-j}}\\
    &\quad-2ne_ne_2 - 2e_2\sum_{j=2}^{n-1}(j-1)e_jX_{P_{n-j}}\\
    &= \frac{1}{4}\sum_{j=2}^{n-3}(j-1)e_jX_{P_{n-j,v,w}}\\
    &\quad+ (n+2) e_{n+2} + n e_{n+1} e_1 + 3 (n-2) e_{n-1} e_3
  + 2(n-3) e_{n-2} e_3 e_1 +4(n-3) e_{n-2} e_4\\
    &\quad-2 e_ne_2 -(n-4)e_{n-2}e_2^2  -(n-2) e_{n-1} e_2e_{1}.
\end{align*}
To rearrange this into the announced form, we peel off the $j=2$ term from the sum and group it with the negative terms:
\begin{align*}
    \frac{1}{4}X_{P_{n,v,w}}&= \frac{1}{4}\sum_{j=3}^{n-3}(j-1)e_jX_{P_{n-j,v,w}}\\
    &\quad+ (n+2) e_{n+2} + n e_{n+1} e_1 + 3 (n-2) e_{n-1} e_3
  + 2(n-3) e_{n-2} e_3 e_1 +4(n-3) e_{n-2} e_4\\
    &\quad+ e_2 \left[\frac{1}{4} X_{P_{n-2,v,w}} -2 e_n -(n-4)e_{n-2}e_2  -(n-2) e_{n-1}e_{1}\right].
\end{align*}

Next we prove that the term within brackets 
\begin{equation}\label{eq: positive P_n-2}
    \left[ \frac{1}{4} X_{P_{n-2,v,w}} -2e_n -(n-4)  e_{n-2}e_2 -(n-2)  e_{n-1}e_1\right]
\end{equation}
is $e$-positive. This follows by comparing the coefficients of the $e$-functions involved. Consider~\eqref{eq:useful_path_twinned_both_leaves} applied to 
$\frac{1}{4}X_{P_{n-2,v,w}}$, and use the coefficients described in~\Cref{cor:coeffS-RECURSION-pathS-cycleS}. We obtain the following formulas for the coefficients:
\begin{equation*}
    [e_n]\tfrac{1}{4}X_{P_{n-2,v,w}} = n, \qquad [e_{n-1}e_1]\tfrac{1}{4}X_{P_{n-2,v,w}} = n-2, \quad \text{and} \quad [e_{n-2}e_2]\tfrac{1}{4}X_{P_{n-2,v,w}} = n-4.
\end{equation*}
In particular, notice that the coefficients of $e_n$ and $e_{n-1}e_1$ are zero and that the coefficient of $e_ne_2$ is $n-2$, which is positive for $n\ge 6$. 
Thus, the negative terms appearing in~\eqref{eq: positive P_n-2} are absorbed by terms in $\frac{1}{4}X_{P_{n-2,v,w}}$, and~\eqref{eq: positive P_n-2} is $e$-positive. 

Finally, as this is an $e$-positive recurrence with $e$-positive initial conditions, by induction it follows that $X_{P_{n,v,w}}$ is $e$-positive for all $n$. 
\end{proof}

\subsubsection{Paths twinned at an interior vertex}\label{sect_recurel_pathinterior}

Next we provide an $e$-positive recurrence relation for path graphs twinned at an interior vertex. 
\begin{theorem}\label{thm:epoS-twin-pathS-recurrence}
The chromatic symmetric function $X_{P_{n,\ell}}$ for the  path $P_n$ twinned at the interior vertex $\ell$ satisfies the $e$-positive recurrence for $\ell\ge 2$, $n\ge \ell +1$, and $n\ge 4$,
 \begin{equation*}\begin{split}X_{P_{n,\ell}} 
        &=\sum_{j=2}^{n
        -\ell-1}(j-1)e_j X_{P_{n-j,\ell}}
        +4(n+1)e_{n+1} +2n e_1 e_n+2e_1\sum_{j=n-\ell+2}^{n-1} (j-1)e_j X_{P_{n-j}}\\ 
        &\quad+4 \sum_{j=n-\ell+3}^n (j-1)e_j X_{P_{n+1-j}}+2 \sum_{j=n-\ell+1}^{n-\ell+2} \textcolor{black}{(j-2)} e_j X_{P_{n+1-j}}
        +(n-\ell-2) e_{n-\ell} X_{P_{\ell,\ell}}.
        \end{split}
    \end{equation*}
Thus, for $n\ge 3$ and $2\le \ell \le n-1$, $X_{P_{n,\ell}}$ is $e$-positive. 
\end{theorem}
\begin{proof}  
Fix $\ell\ge 2$, and consider $n\ge \ell +1$ with $n\ge 4$. We start with the recurrence relation  in the statement of~\Cref{JM}. 
This can be rewritten as
\begin{equation}\label{eqn:Reorder-twin-path-deg2-recS-v2}
\begin{split}
X_{P_{n,\ell} }
&
= 4X_{P_{n+1}}+ 2e_1 X_{P_n} +2e_2 X_{P_{\ell - 1}}X_{P_{n -\ell}}\\
&
-2X_{P_{\ell+1}} X_{P_{n - \ell}} - 2X_{P_{\ell}} X_{P_{n - \ell+1}}
-2X_{P_{\ell - 1}}X_{P_{n - \ell +2}}.
\end{split}
\end{equation} 
 
From~\eqref{eqn:Reorder-twin-path-deg2-recS-v2} and using ~\Cref{prop:RECURSION-pathS-cycleS}\ref{RECURSION-pathS}, we have that
\begin{align}X_{P_{n,\ell}} 
&=4 (n+1) e_{n+1} + 4\sum_{j=2}^n (j-1)e_j X_{P_{n+1-j}}
+2 e_1 \left[n\, e_n + \sum_{j=2}^{n-1} (j-1)e_j X_{P_{n-j}}\right]\notag \\
&\quad+ 2 e_2 X_{P_{\ell-1}}  \left[ (n-\ell) e_{n-\ell} +\sum_{j=2}^{n-\ell-1} (j-1) e_j  X_{P_{n-\ell-j}}\right]\notag \\
&\quad-2 X_{P_{\ell+1}}  \left[{{(n-\ell) e_{n-\ell}}}+ \sum_{j=2}^{n-\ell-1} (j-1) e_j  X_{P_{n-\ell-j}}\right]\notag \\
&\quad-2 X_{P_{\ell}}  \left[ {{(n+1-\ell) e_{n+1-\ell}}}+ \sum_{j=2}^{n-\ell} (j-1) e_j  X_{P_{n+1-\ell-j}}\right]\label{eqn:int-twin-temp1}\\
&\quad-2 X_{P_{\ell-1}}  \left[ {{(n+2-\ell) e_{n+2-\ell}}}+ \sum_{j=2}^{n-\ell+1} (j-1) e_j  X_{P_{n+2-\ell-j}}\right].\label{eqn:int-twin-temp2}
\end{align}
Notice that, in the six summands above, for each fixed $j$, the terms attached to the factor $(j-1)e_j$, when collected together, match the six terms in the right-hand side of the recurrence~\eqref{eqn:Reorder-twin-path-deg2-recS-v2} applied to $X_{P_{n-j,\ell}}$.
Grouping the remaining terms into an expression $Y$ if they have a positive sign, or  $Z$ if they have a negative sign, we obtain
\begin{equation*}
X_{P_{n,\ell}} = {\sum_{j=2}^{n-\ell-1} (j-1) e_j  X_{P_{n-j,\ell}}  }
+Y { {-Z}}.
\end{equation*}
The positive terms $Y$ are given by 
\begin{equation*}\label{eqn:twin-path-int-Y}
\begin{split}
Y&= 4 (n+1) e_{n+1}+2 ne_n e_1 +  {2 e_2 X_{P_{\ell-1}} (n-\ell) e_{n-\ell}} \\
&\quad+4 \sum_{j=n-\ell}^n (j-1)e_j X_{P_{n+1-j}}+2e_1 \sum_{j=n-\ell}^{n-1} (j-1)e_j X_{P_{n-j}}\\
&=4 (n+1) e_{n+1}+2 n e_n e_1 
+4\!\!\! \sum_{j=n-\ell+3}^n\!\!\! (j-1)e_j X_{P_{n+1-j}} +2e_1\!\sum_{j=n-\ell+2}^{n-1} (j-1)e_j X_{P_{n-j}}  \\
&\quad+ \underbrace{4\!\!\sum_{j=n-\ell}^{n-\ell+2}\!\! (j-1)e_j X_{P_{n+1-j}}+ 2e_1\!\!\sum_{j=n-\ell}^{n-\ell+1}\!\! (j-1)e_j X_{P_{n-j}}+ 2 e_2 X_{P_{\ell-1}} (n-\ell) e_{n-\ell}}_{Y_1},
\end{split}
\end{equation*}
where the last line is obtained by splitting the summations. 
The negative terms $Z$ are given by
\begin{align*}
 Z&= { {2 X_{P_{\ell+1}}  (n-\ell) e_{n-\ell} +2 X_{P_{\ell}}   (n+1-\ell) e_{n+1-\ell}  +2 X_{P_{\ell-1}}  (n+2-\ell) e_{n+2-\ell}}}\\
&\quad+ 2 X_{P_{\ell}} (n-\ell-1) e_{n-\ell} X_{P_1} + 
2 X_{P_{\ell-1}} \sum_{j=n-\ell}^{n-\ell+1} (j-1) e_j  X_{P_{n+2-\ell-j}},
\end{align*}
where the last two terms come from the sums  in~\eqref{eqn:int-twin-temp1} and~\eqref{eqn:int-twin-temp2}. We rewrite $Z$ so that $Y-Z$ is easier to analyze.
\begin{equation*}
\begin{split}
Z&={ { 2\sum_{j=n-\ell}^{n-\ell+2} j e_j X_{P_{n+1-j}}}}\\
&\quad+{ { 2 X_{P_{\ell}} (n\!-\!\ell\!-\!1) e_{n-\ell} X_{P_1} + 
2 X_{P_{\ell-1}}(n-\ell) e_{n-\ell+1}X_{P_1}}}
+ {2 X_{P_{\ell-1}} (n\!-\!\ell\!-\!1) e_{ n-\ell}  X_{P_{2}}}\\
&=2\sum_{j=n-\ell}^{n-\ell+2} j e_j X_{P_{n+1-j}}
+ 2e_1 \sum_{j=n-\ell}^{n-\ell+1} (j-1)e_j X_{P_{n-j}} 
+ 2 X_{P_{2}}X_{P_{\ell-1}} (n\!-\!\ell\!-\!1) e_{ n-\ell}
\end{split}
\end{equation*}
Using $X_{P_2}=2e_2$, 
we have 
\[ Y_1-Z={ { 2\sum_{j=n-\ell}^{n-\ell+2} (j-2) e_j X_{P_{n+1-j}}}}
- {2 X_{P_{\ell-1}} (n-\ell-2) e_{ n-\ell}  e_2},\]
because the terms with the factor $e_1=X_{P_1}$ can be seen to vanish identically. 
By splitting the sum, $Y_1-Z$ can be rewritten as 
\[ Y_1-Z= 2\sum_{j=n-\ell+1}^{n-\ell+2} (j-2) e_j X_{P_{n+1-j}}
+ 2 (n-\ell-2)  e_{ n-\ell}  {(X_{P_{\ell+1}}- e_2 X_{P_{\ell-1}})}.\]
In the last term on the right-hand side, the factor of $2(X_{P_{\ell+1}}- e_2 X_{P_{\ell-1}})=2X_{P_{\ell+1}}- X_{P_2} X_{P_{\ell-1}}=X_{P_{\ell,\ell}}$ is precisely the chromatic symmetric function of the path $P_\ell$ twinned at a leaf.  Thus, putting all of this together, we obtain the recurrence relation from the statement. 

Finally, we can deduce the $e$-positivity. For the initial values $n=\ell+1, \ell+2, \ell+3$, $X_{P_{n, \ell}}$ is  $e$-positive by~\Cref{theorem:twin-path-epoS-vertex}.
We proceed by strong induction on $n$ to show the claimed $e$-positivity for $X_{P_{n, \ell}}$ for $n\geq \ell+4$. Our induction hypothesis is that $X_{P_{m,\ell}}$ is $e$-positive for all $m<n, m\ge \ell+1$. We only need to look at $Y_1-Z$ since that is the part containing negative terms. By~\Cref{SHORTprf-epoS-twin-path-leaf} we know that $Y_1-Z$ is $e$-positive.  Hence,  
by induction the proof is complete.
\end{proof}

As an application of the preceding recurrence, we have the following corollary.

\begin{cor}[{See also~\cite[Theorem 7.8]{gebhard_sagan_2001}}]\label{cor:SHORTprf-epoS-double-twin-int-vertex-leaf} 
Consider the path $P_n$ on $n$ vertices, with $n\ge 4$. Let $2\le \ell \le n-2$ and let $v$ be the leaf $n$. 
Then the chromatic symmetric function of $P_{n,\ell,v}$ is $e$-positive.
\end{cor}
\begin{proof} The triple deletion argument implies that  
\[X_{P_{n,\ell,v}}=2 X_{P_{n+1,\ell}}- X_{P_2} X_{P_{n-1,\ell}}=
2(X_{P_{n+1,\ell}}-e_2 X_{P_{n-1,\ell}}).\]
Examining the recurrence for $X_{P_{n,\ell}}$ in ~\Cref{thm:epoS-twin-pathS-recurrence} one sees that, when $n-\ell-1\ge 2$,  the initial term in the first sum in the  expression for $X_{P_{n,\ell}}$ is $e_2 X_{P_{n-2,\ell}}$, making $X_{P_{n,\ell}}-e_2 X_{P_{n-2,\ell}}$  $e$-positive.  Replacing $n$ by $n+1$ now gives $e$-positivity of $X_{P_{n+1,\ell}}-e_2 X_{P_{n-1,\ell}}$ for 
$n-\ell\ge 2$.
\end{proof}

\subsection{Recurrence for twinned cycles} 

In this section we derive an $e$-positive  recursive formula for  the  twinned cycle, analogous to those for the twinned path from the last section.  We give a similar formula for another family of graphs that we call  moose graphs.

    \subsubsection{Twinned cycles}
We start with the cycle graph.
    \begin{theorem}\label{thm:SHORTprf-epoS-twin-cycleS-recurrence} The chromatic symmetric function $X_{C_{n,v}}$ for the  cycle $C_n$ twinned at a vertex $v$ is $e$-positive. 
For $n\geq 5$, it satisfies the $e$-positive recurrence 
\begin{equation*}
\begin{split}
X_{C_{n,v}} &=\sum_{k=3}^{n-2}(k-1)e_k X_{C_{n-k,v}} 
+2(n+1)(2n-3)e_{n+1}+2(n-1)(n-3)e_{n} e_1 \\
&\quad+ e_2\left[  X_{C_{n-2,v}} -2(n-3)e_{n-1}\right],
\end{split}
\end{equation*}
with initial conditions
$$
X_{C_{1,v}}=2 e_2,\qquad X_{C_{2,v}}=6 e_3,\qquad X_{C_{3,v}}=24 e_4, \quad \text{and} \quad X_{C_{4,v}}=50 e_5 +6 e_4 e_1 +4 e_3 e_2. $$   
\end{theorem}

\begin{proof}
We proceed by induction on $n$. The initial cases $n \leq 4$ are verified by direct computation using~\Cref{thm:RPS-gf-paths-cycles} and~\Cref{lemma:twin-cycle-TD}. Note that these initial terms are all $e$-positive.

Assume we have shown the claim to be true for $X_{C_{m,v}}$ for all $m < n$. Now we rewrite the expression in~\Cref{lemma:twin-cycle-TD} using~\Cref{prop:RECURSION-pathS-cycleS} to obtain:
\begin{align}
X_{C_{n,v}}&=4(n+1) n e_{n+1} +4 \sum_{k=2}^{n-1} (k-1) e_k X_{C_{n+1-k}}\notag\\
&\quad+2e_1\left[n(n-1) e_n +\sum_{k=2}^{n-2} (k-1) e_k X_{C_{n-k}} \right]\notag\\
&\quad-6\left[(n+1) e_{n+1} +\sum_{k=2}^{n} (k-1) e_k X_{P_{n+1-k}} \right]\notag\\
&\quad+2e_2\left[ (n-1) e_{n-1} +\sum_{k=2}^{n-2} (k-1) e_k X_{P_{n-1-k}} \right].\label{eq:X_CnvFirstRecurrence}
\end{align}

Applying~\Cref{lemma:twin-cycle-TD} again,  we can collect the four summations above into one sum and three additional terms as follows:
\begin{align*}
&\sum_{k=2}^{n-2}(k-1)e_k\left[4X_{C_{n+1-k}} + 2e_1 X_{C_{n-k}} - 6 X_{P_{n+1-k}} + 2e_2X_{P_{n-1-k}}\right] \\
&\quad+  4(n-2)e_{n-1}X_{C_2} - 6(n-2) e_{n-1}X_{P_2} - 6(n-1)e_nX_{P_1} \\
&=\sum_{k=2}^{n-2} (k-1)e_kX_{C_{n-k,v}} + 4(n-2)e_{n-1}X_{C_2} - 6(n-2) e_{n-1}X_{P_2} - 6(n-1)e_nX_{P_1}.
\end{align*}

Combining this expression with the remaining terms  from~\eqref{eq:X_CnvFirstRecurrence}, we obtain
\begin{align}
X_{C_{n,v}} &=\sum_{k=2}^{n-2}(k-1)e_k X_{C_{n-k,v}} 
+2(n+1)(2n-3)e_{n+1}+2(n-1)(n-3)e_{n} e_1 \notag\\
&\quad-2(n-3)e_{n-1}e_2.
\label{eqn:REC-twin-cycleS}
\end{align}

We isolate the term $k=2$ from the summation and regroup it with the last term in~\eqref{eqn:REC-twin-cycleS}, so that it becomes
 \begin{align*}
 X_{C_{n,v}} &=\sum_{k=3}^{n-2}(k-1)e_k X_{C_{n-k,v}}  +2(n+1)(2n-3)e_{n+1}  +2(n-1)(n-3)e_{n} e_1\\
 &\quad  +e_2\left[X_{C_{n-2,v}}-2(n-3)e_{n-1}\right],
 \end{align*}
 as stated in the theorem. 

 Now we want to show $e$-positivity. By the induction hypothesis, only the last term requires scrutiny.  Using ~\Cref{cor:coeffS-RECURSION-pathS-cycleS} and~\Cref{lemma:twin-cycle-TD}, the coefficient of $e_{n-1}$ in $X_{C_{n-2,v}}$ is $2(n-1)(2n-7)$. 
Therefore, the coefficient of $e_{n-1}$ in 
$X_{C_{n-2,v}}-2(n-3)e_{n-1}$
is $2(n-1)(2n-7) - 2(n-3) = 4(n^2 -5n + 5)$, which is nonnegative for $n \geq 4$.  Thus, by the induction hypothesis, the formula in the statement for $X_{C_{n,v}}$ is indeed an $e$-positive recurrence for the twinned cycles for $n \geq 5$. 
\end{proof}

    \subsubsection{The moose graph}
 
We define the \textit{moose graph} $A_{n+2}$ to be the graph on $n+2$ vertices and $n+1$ edges, obtained from the cycle graph $C_n$ by attaching a leaf to each of the vertices $v$, $w$ of an edge $v w$ in $C_n$,

\begin{figure}[ht]
    \centering
    \begin{tikzpicture}[scale=.7]
\draw(1,1.7) -- (1,3.5);
\draw(1,1.7) -- (2,0);
\draw(2,0) -- (1,-1.7);
\draw(1,-1.7) -- (0.6,-2);
\draw(-1,-1.7) -- (-0.6,-2);
\draw(-1,-1.7) -- (-2,0);
\draw(-2,0) -- (-1,1.7);
\draw(-1,1.7) -- (1,1.7);
\draw(-1,1.7) -- (-1,3.5);
\node at (0.1,-2.1){$\cdots$};
\node(1') at (1,3.5){$\bullet$};
\node(1) at (1,1.7){$\bullet$};
\node(2) at (2,0){$\bullet$};
\node(3) at (1,-1.7){$\bullet$};
\node(4) at (-1,-1.7){$\bullet$};
\node(5) at (-2,0){$\bullet$};
\node(6) at (-1,1.7){$\bullet$};
\node(6') at (-1,3.5){$\bullet$};
\node[left, xshift = -2pt] at (6){$v$};
\node[right, xshift = 2pt] at (1){$w$};
\end{tikzpicture}
    \caption{The moose graph $A_{n+2}$.}
    \label{fig:Moose}
\end{figure}
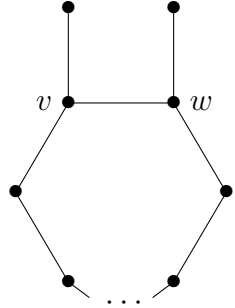

We provide an $e$-positive recurrence relation for the chromatic symmetric function of the moose graph. This graph was shown to be $e$-positive as a special case in \cite[Theorem 3.9]{WangZhou2024}.  We omit the proof. 

\begin{prop}\label{thm:MooseRecurrence}
For $n\geq 2$, the chromatic symmetric function of the moose graph $A_{n+2}$ is $e$-positive. For $n\ge 4$, it satisfies the $e$-positive recurrence 
\begin{align*}
X_{A_{n+2}}&=  \left(\sum_{j=2}^{n-2} (j-1)e_j X_{A_{n+2-j}}\right)\\
&\quad+(n+2)(n-1)e_{n+2} +2e_1 e_{n+1} (n^2-n-1) +(n-1)(n-2) e_1^2 e_n +{2 e_2e_n}, 
\end{align*}
with initial values
\begin{align}X_{A_4}&=X_{P_4}=2e_2^2+2e_3 e_1+ 4 e_4, \nonumber \\ X_{A_5}&=2e_3 e_1^2 
+ 2 e_3 e_2 + 10 e_4 e_1+ 10 e_5, \nonumber \\ X_{A_6}&=2 e_2^3 + 2 e_3 e_2 e_1 + 6  e_4 e_1^2 + 6 e_4 e_2 + 22 e_5 e_1  + 18 e_{6}.  \nonumber 
\end{align}
\end{prop}

\section{Future directions}

This work provides explicit $e$-positive generating functions for the chromatic symmetric function of twinned paths and cycles, and suggests that it may be worthwhile to undertake a similar study for twins of other graph families.

More generally, an examination of \Cref{table:known_results} shows that much of the recent literature focuses on establishing Gebhard and Sagan's $(e)$-positivity of the chromatic symmetric function in \emph{noncommuting} variables. Although $(e)$-positivity implies $e$-positivity as a symmetric function in ordinary commuting variables, in such cases an explicit $e$-positive generating function or recurrence would be desirable. We propose the following future investigations in this direction: 

\begin{enumerate}
    \item For the twinned cycle graph, is the chromatic symmetric function in noncommuting variables  $(e)$-positive?
    \item Are there pleasing $e$-positive symmetric function expansions for those families whose $e$-positivity is known only via the stronger $(e)$-positivity property? Specific examples that may admit nice generating functions are the triangular ladder   \cite{dahlberg2019triangular,stanley_1999}, the kayak paddle graphs \cite{aliniaeifard2021chromatic} and the tadpole graph \cite{gebhard_sagan_2001, Li-Li-Wang-Yang}.
\end{enumerate}

It  would also be interesting to examine twinning for the \textit{chromatic quasisymmetric function} of Shareshian and Wachs~\cite{SW,shareshian_wachs_2016} since the main class of posets of study for these, whose incomparability graphs are unit interval graphs, 
is also closed under the appropriately defined twinning operation for labeled graphs. 

\subsection*{Acknowledgments}
This work began at the 2023 Graduate Research Workshop in Combinatorics, which was supported in part by NSF grant \#1953445, NSA grant \#H98230-23-1-0028, and the Combinatorics Foundation. We would like to thank the organizers for providing a great venue for collaboration.
We also gratefully acknowledge Spencer Daugherty's early contributions to this project.

E. Banaian was supported by Research Project 2 from the Independent Research Fund Denmark (grant no. 1026-00050B). J. Lentfer was supported by the National Science Foundation Graduate Research Fellowship DGE-2146752. 

\bibliographystyle{amsplain}

\bibliography{TwinningPaperArXiv2024Nov3}

\end{document}